\documentclass[english,reqno,11pt]{amsart}
\usepackage{amsmath, amsthm, amsfonts}
\usepackage{hyperref}
\hypersetup{
	colorlinks=true,
		citecolor=blue!60!black,
		linkcolor=red!60!black,
		urlcolor=green!40!black,
		filecolor=yellow!50!black,
	breaklinks=true,
	pdfpagemode=UseNone,
	bookmarksopen=false,
}
 
\usepackage{amsmath,amsthm,amssymb}
\usepackage{amscd,indentfirst,epsfig}
\usepackage{latexsym}
\usepackage{times}
\usepackage{enumerate}
\usepackage{mathrsfs}
\usepackage{stmaryrd}
\usepackage{amsopn}
\usepackage{amsmath}
\usepackage{amssymb,dsfont,mathtools}
\usepackage{amsfonts,bm}
\usepackage{amsbsy,amsmath}
\usepackage{amscd}
\usepackage{xcolor}
\usepackage[mono=false]{libertine}
\usepackage[T1]{fontenc}
\usepackage{amsthm}
\usepackage[normalem]{ulem} 
\usepackage[cal=euler, scr=boondoxo]{}
\usepackage{microtype}

\usepackage{numprint}

\makeatletter
\def \leq {\leqslant}
\def \le {\leq}
\def \geq {\geqslant}
\def\e{\varepsilon}
\def \ge {\geq}

\def\R{\mathbb R}
\def\S{\mathbb S}

\def\N{\mathbb N}

 \def \la {\langle}
 \def \ra {\rangle}

\def\l{\ell}

\def\g{\gamma}
\def \ds {\displaystyle}

\def \d {\mathrm{d}}
\def \lm {\bm{m}}
\def \lM {\mathds{M}}
\def \lD {\mathds{D}}

\def \Q {\mathcal{Q}}

\def\fet{f_{\ast}}

\def\vet{v_{\ast}}

\def\grad{\nabla}

\newtheorem{theo}{Theorem}[section]

\newtheorem{prop}[theo]{Proposition}

\newtheorem{lem}[theo]{Lemma}

\newtheorem{rmq}[theo]{Remark}

\newcommand{\vertiii}[1]{{\left\vert\kern-0.25ex\left\vert\kern-0.25ex\left\vert #1  
    \right\vert\kern-0.25ex\right\vert\kern-0.25ex\right\vert}}                      
\newcommand{\verti}[1]{{\left\vert\kern-0.25ex\left\vert\kern-0.25ex\left\vert #1    
    \right\vert\kern-0.25ex\right\vert\kern-0.25ex\right\vert}}	

\def \ind {\mathbf{1}}
\numberwithin{equation}{section}

\usepackage[many]{tcolorbox}
  
\title[]{Conditional integrability and stability for the homogeneous Boltzmann equation with very soft potentials}
\author{R. Alonso}
\address{Texas A\&M University at Qatar, Division of Arts and Sciences, Education City, Doha, Qatar.} \email{ricardo.alonso@qatar.tamu.edu}

\author{P. Gervais}
\address{Universit\'e Lille, Inria, CNRS, UMR 8524 - Laboratoire Paul Painlev\'e, F-59000 Lille, France}\email{pierre.gervais@inria.fr}

\author{B.  Lods}
\address{Universit\`{a} degli
Studi di Torino \& Collegio Carlo Alberto, Department of Economics, Social Sciences, Applied Mathematics and Statistics ``ESOMAS'', Corso Unione Sovietica, 218/bis, 10134 Torino, Italy.}\email{bertrand.lods@unito.it}

\begin{document}

\maketitle
 
\begin{abstract} We introduce a practical criterion that justifies the propagation and appearance of $L^{p}$-norms for the solutions to the spatially homogeneous Boltzmann equation with very soft potentials without cutoff.  Such criterion also provides a new conditional stability result for classical solutions to the equation. All results are quantitative. Our approach is inspired by a recent analogous result for the Landau equation derived in \cite{ABDL-PS} and generalises existing conditional results related to higher integrability properties and stability of solutions to the Boltzmann equation with very soft potentials.



\end{abstract}

\tableofcontents

\section{Introduction}

\subsection{The spatially homogeneous Boltzmann equation} The Boltzmann equation describes the evolution of the density $f=f(t,v)\geq0$ of gas at time $t\geq0$ with particle velocity $v\in \R^{d}$.  It is given by
\begin{subequations}\label{eq:Boltz}
\begin{equation}\label{eq:Boltz1}
\partial_{t} f(t,v) = \Q(f(t,\cdot),f(t,\cdot))(v),  
\qquad t\geq0, \quad v \in \R^{d}, \quad d \ge 2,
\end{equation}
supplemented with the initial condition
\begin{equation}\label{eq:Init}
f(t=0,v)=f_{\mathrm{\rm in}}(v), \qquad v \in \R^{d}.\end{equation}
\end{subequations}
The integral operator $\Q$ denotes the Boltzmann collision operator which is given by
\begin{equation}\label{op:Bol}
\Q(g,f)(v) := \int_{\R^{d}}\d \vet \int_{\S^{d-1}}\bigg(g(\vet')f(v')-g(\vet)f(v)\bigg)B(v-\vet, \sigma)\d\sigma\,.
\end{equation}
Here above, $\sigma \in \S^{d-1}$ is the scattering direction after a particle-particle collision with velocities $v$ and $\vet \in \R^{d}$.  The post collisional particle velocities $v'$ and $v'_*$ are given by
$$v'=v-u^{-}, \qquad \vet'=\vet+u^{-},\quad u=v-\vet, \quad u^{\pm}=\frac{u\pm |u|\sigma}{2}, \qquad \widehat{u}=\frac{u}{|u|}.$$
In this document we work with a collisional kernel $B(u,\sigma)$ that take the form
$$B(u,\sigma)=|u|^{\g}b(\widehat{u}\cdot \sigma)=|u|^{\g}b(\cos\theta), \qquad \cos \theta=\widehat{u} \cdot \sigma\,, \quad u \in \R^{d},\quad \sigma \in \S^{d-1}\,.$$
Moreover, in the present contribution, we consider the so-called without cutoff assumption
\begin{subequations}\label{eq:Noyau}
\begin{equation}\label{eq:b0s}\sin^{d-2}(\theta)b(\cos \theta) \sim b_{0}\,\theta^{-1-2s}, \qquad \text{ as } \quad \theta \sim 0\,,\end{equation}
for some $b_{0}> 0$ in the regime of very soft potentials
\begin{equation}\label{eq:g2s}
\g+2s < 0, \qquad s \in (0,1), \qquad \g > \max\{-4,-d\}.
\end{equation}
\end{subequations}
{For this range the existence of \textit{H-weak solutions} and the linear-in-time propagation of statistical moments is guaranteed, see \cite{carlen}.}  Such kernels are derived from long-range potentials obeying inverse power law $U(r)=r^{1-\mathfrak{p}}$ with $2 < \mathfrak{p} < \infty$ {for the physical three dimensional case}.  In this case,
$$-3 < \g=\frac{\mathfrak{p}-5}{\mathfrak{p}-1} < 1, \qquad 0 < s=\frac{1}{\mathfrak{p}-1} < 1.$$
Notice that, $\g=(\mathfrak{p}-5)s$, and $\g+2s=(\mathfrak{p}-3)s$, so that we are assuming $2  < \mathfrak{p} < 3$.
It is well known that due to the large forward scattering of collision, that is for small scattering angles $\theta\sim0$, a diffusion-like regularisation effect is expected in this case.\medskip

We notice that in the case of hard or moderately soft potentials, that is $\g + 2s \ge 0$, it is known that global smooth solutions exist \cite{Im-Sil}. However, for the case \eqref{eq:g2s} addressed in this work smoothness and uniqueness of solutions to \eqref{eq:Boltz} is an open problem.  Working in this direction,  our objective is to provide a conditional result aiming to:
\begin{enumerate}[1)]
\item showing the appearance of $L^{p}$-integrability of solutions to \eqref{eq:Boltz} as well as the propagation of such integrability if initially assumed.
\item proving the stability of solutions to \eqref{eq:Boltz} in suitable weighted $L^{2}$-spaces and, thus, implying the uniqueness of solutions in such spaces.
\end{enumerate}
The precise statement of these results is in Theorem \ref{mps_g} below.  Let us introduce some useful notation used in the sequel.
\subsection{Notations} For $k \in \R $ and $ p\geq 1$, we define the Lebesgue space {$L^{p}_{k}=L^{p}_{k}(\R^{d})$} through the norm
$$\displaystyle \|f\|_{L^p_{k}} := \left(\int_{\R^{d}} \big|f(v)\big|^p \, 
\langle v\rangle^{k} \, \d v\right)^{\frac{1}{p}}, \qquad L^{p}_{k}(\R^{d}) :=\Big\{f\::\R^{d} \to \R\;;\,\|f\|_{L^{p}_{k}} < \infty\Big\}\, ,$$
where $\langle v\rangle :=\sqrt{1+|v|^{2}}$.  In the particular case $k=0$, we simply denote $\|\cdot\|_{L^p}$ the $L^{p}$-norm. For any $\alpha \geq0$, we define the homogeneous Sobolev spaces $\dot{\mathbb{H}}^{\alpha}$ through the norm
\begin{equation}\label{eq:notaHs} 
\|g\|_{\dot{\mathbb{H}}^{\alpha}}=\left[\int_{\R^{d}}\left|\widehat{g}(\xi)\right|^{2}\,|\xi|^{2\alpha}\d\xi\right]^{\frac{1}{2}}, \qquad \widehat{g}(\xi)=\int_{\R^{d}}\exp\left(-i x \cdot \xi\right)g(x)\d x, \qquad \xi \in \R^{d}.
\end{equation}
and the Sobolev space $\mathbb{H}^{\alpha}$ through the norm $\|\cdot\|_{\mathbb{H}^{\alpha}}^{2}=\|\cdot\|_{\dot{\mathbb{H}}^{\alpha}}^{2}+\|\cdot\|_{L^{2}}^{2}$.  
We refer to \cite{chemin} for more details about Sobolev spaces. We will repeatedly use the fact that, for $k_{1}\geq k_{2} \geq0,$ it holds
$$\|\la\cdot\ra^{k_{2}}g\|_{\mathbb{H}^{\alpha}} \leq C\|\la \cdot \ra^{k_{1}}g\|_{\mathbb{H}^{\alpha}}, \qquad \alpha \geq0$$
for some positive constant $C >0$ depending only on $d,\alpha,k_{1}-k_{2}.$ Notice that such an estimate is not true for the norm associated to the homogeneous space $\dot{\mathbb{H}}^{\alpha}.$ Given $k \in \R$ and $f \in L^{1}_{k}(\R^{d})$, we also define the statistical moments as
$$
\lm_{k}(f)=\int_{\R^{d}}f(v)\langle v\rangle^{k}\d v\,.
$$
Given mass $\varrho >0$, energy $E >0$, and entropy bounds $H >0$, we introduce the subclass $\mathcal{Y}(\varrho,E,H)$ of $L^{1}_{2}(\R^{d}) \cap L\log L(\R^{d})$ as 
\begin{multline}\label{eq:classY}
\mathcal{Y}(\varrho,E,H)=\left\{g \in L^{1}_{2}(\R^{d})\,;\,g \geq0\,;\,\lm_{0}(g) \geq \varrho,\,\,\lm_{2}(g) \leq E\,\right.\\
\left.\,\int_{\R^{d}}g(v)\log g(v)\d v \leq H\right\}.
\end{multline}
This class is a natural space for solutions of the Boltzmann equation.  In the sequel, we will \emph{fix} $\varrho_{\rm in} >0, E_{\rm in} >0, H_{\rm in} >0$ and will use the shorthand notation
$$\mathcal{Y}_{\rm in}=\mathcal{Y}(\varrho_{\rm in},E_{\rm in},H_{\rm in})\,.$$
The initial datum $f_{\rm in}$ associated to \eqref{eq:Boltz} is always assumed to belong to $\mathcal{Y}_{\rm in}$.

Finally, for a measurable mapping $b\,:\,(-1,1) \to \R^{+}$, we identify the function $b$ to a function defined over the unit sphere $\S^{d-1}$ through the identification $\sigma \in \S^{d-1} \mapsto \cos \theta \in (-1,1)$ and write
\begin{align*}
\|b\|_{L^{1}(\S^{d-1})}&=\int_{\S^{d-1}} b(\widehat{u} \cdot \sigma)\d \sigma =\big| \S^{d-2}\big|\int^{\pi}_{0}b(\cos\theta)(\sin\theta)^{d-2}\d\theta\, \qquad \forall \,\widehat{u}  \in \S^{d-1}.
\end{align*}
\subsection{Main results and related literature} In the context of the spatially homogeneous Boltzmann equation without cutoff, the question of the regularity and well-posedness of solutions is well understood in the case of hard potentials corresponding to $\g >0$ \cite{Im-Sil,ricardo} and moderately soft potentials $\g+2s \geq0$ \cite{Im-Sil}.  Global existence of classical smooth solutions is known when $\g+ 2s \geq 0$, see for example \cite{DM,He12,Im-Sil}.  Additionally, measure-valued solutions are known to exist globally when $\g \geq -2$, see \cite{Lu,Mori16}, whereas a very weak notion of solution, based on the entropy dissipation functional, have been constructed in \cite{villH} under the name of $H$-solutions; these solutions exist for a large range of potentials and propagate statistical moments, namely  {$\g \geq -4$} and $0 < s < 1$ as proven in \cite{carlen}.  A new functional inequality for the entropy production derived in \cite{chaker} has shown that $H$-solutions are in fact weak solution for $\g+2s > -2$.  

This work is aimed to discuss the delicate case of very soft potentials for which $\g+2s < 0$ with $\gamma>\max\{-4,-d\}$ and $0<s<1$.  For such negative values of $\g$, the singularities of the collision kernel $B(v-\vet,\sigma)$ in \eqref{op:Bol} are severe and constrain solutions to possess a minimal threshold of integrability and smoothness to be meaningful, at least in a weak sense.  Currently, such minimal threshold has not been proven in an unconditional manner.  It is the aim of this document to provide a conditional integrability condition on solutions that allows higher integrability which, in turn, results in an stability theorem for weak solutions.  Interestingly, for suitable modification of the Boltzmann equation, introduced recently in \cite{snelson}, higher integrability conditions have been unconditionally established yielding a satisfactory well-posedness theory for such a ``isotropic'' Boltzmann model.

For the very soft potentials case, the conditional regularity of solutions to the Boltzmann equation without cutoff has been introduced in the work \cite{Silv} where it is shown that if the solution lies in $L^{\infty}([0,T]; L^{p}_{k}(\R^3))$ for $p$ and $k$ sufficiently large, then it is bounded.  In addition, a suitable rate of appearance of pointwise bounds is provided. More precisely, for a suitable solution $f : (0, T ) \times \R^{d}\to\R$ to \eqref{eq:Boltz} and $p \in \left(\frac{d}{d+\g+2s},\frac{d}{d+\g}\right)$ and $k >0$ large enough,
\begin{equation}\label{eq:Silv}
f(t) \in L^{\infty}((0,T);L^{p}_{k}(\R^{d})) \implies \|f(t)\|_{\infty} \leq C\left(1+t^{-\frac{d}{2sp}}\right)\end{equation}
for some $C >0$ explicitly depending on $\sup_{t\in [0,T]}\|f(t)\|_{L^p}$. 

Attacking the problem from a different angle, as far as regularity of solutions is concerned, a recent result of \cite{GIS} provides a bound on the Hausdorff dimension of the times in which the solution could be singular. Such a result is the analogue of the celebrated \cite{CKN} result for the Navier-Stokes equation that has been recently adapted also to the Landau equation, see Section \ref{sec:Landau} for details. 
 
Let us state the first theorem that establishes a new conditional criteria ensuring the appearance or propagation of $L^{p}$-norms which, arguably, generalises the one in \eqref{eq:Silv}.  Such criterion follows the spirit of the well-known Prodi-Serrin criteria that guarantees the uniqueness of solutions to the Navier-Stokes equation, see the recent result for the Landau equation \cite{ABDL-PS}. {The theorem is stated for classical solutions, in the sense that we work with solutions where all computations can be carry on meaningfully, and must be understood in the context of \textit{a priori estimates}.}
 
\begin{theo}\label{mps_g}
 
 Let $f_{\rm in} \in \mathcal{Y}_{\rm in}$ be given and let $f=f(t,v)$ be a classical solution to Eq. \eqref{eq:Boltz} with 
 $$B(u,\sigma)=|u|^{\g}b(\cos \theta)$$
where $b(\cdot)$ satisfies  \eqref{eq:b0s}--\eqref{eq:g2s}.   
Assume that the solution $f$ satisfies 
\begin{equation}\label{eq:ProSer1}
\langle \cdot \rangle^{|\g|}\,f \in L^{r}\big([0,T]\,;\,L^{q}(\R^{d})\big) \:\: \text{ with }\:\: \frac{2s}{r}+\frac{d}{q}=2s+d+\g\end{equation} 
where $r \in (1,\infty)$ and $q \in \left(\max\left(1,\frac{d}{d+\g+2s}\right), \infty \right)$.  Then, given $p \in (1,\infty)$, the following integrability properties hold:

\begin{enumerate}[(a)]
	\item \textnormal{\textbf{(Propagation of Lebesgue norm)}} If $f_{\rm in} \in L^{p}(\R^{d})$ then
	\begin{equation}\label{eq:propaLp1}
		\sup_{t \in [0,T]}\|f(t)\|_{L^p} \leq \exp\left\{\frac{\bm{C}_{1}}{p}\int_{0}^{T}\left(1+ \|\langle \cdot\rangle^{|\g|}f(t)\|_{L^q}^{r} \right)\d t\right\}\|f_{\rm in}\|_{L^p}\end{equation}
	for some explicit constant $\bm{C}_{1}$ depending on $p,q,s,\g,d$ and $\varrho_{\rm in},E_{\rm in},H_{\rm in}$ but not on $T$ or $\|f_{\rm in}\|_{L^p}$.\\
	
	\item \textnormal{\textbf{(Appearance of Lebesgue norm)}} If moments for $f_{\rm in}$ are assumed, that is
	$$f_{\rm in} \in L^{1}_{\eta_{p}}(\R^{d}), \qquad \eta_{p}=\eta_{p}(s,\g):=\frac{|\g|d}{2s}\left(1-\frac{1}{p}\right),$$
	then, the solution $f=f(t,v)$ satisfies  the following estimate for $t \in (0,T]$
	\begin{equation}\label{Lp-Kpq}
		\|f(t)\|_{L^p} \le \bm{K}_{p,B,q,T} \,\, t^{- \frac{d}{2s} \left(1 - \frac1{p} \right)}\sup_{\tau \in [0,T]}\lm_{\eta_{p}}(\tau),\end{equation}
	where $\bm{K}_{p,B,q,T}$ is explicitly depending on $\ds\int_{0}^{T} \|\langle \cdot\rangle^{|\g|}f(t)\|_{L^q}^{r} \d t$ and also on $p,q,s,\g,d$ and $\varrho_{\rm in}$, $E_{\rm in}$, $H_{\rm in}$.
\end{enumerate}
\end{theo}
We do not elaborate on the appearance of pointwise bounds in the present contribution, however, the appearance of $L^{p}$-bounds for sufficiently large $p>1$ provides an alternative criterion to the result of \cite{Silv} recalled in \eqref{eq:Silv}. Indeed, as long as solutions possess such higher integrity, it is possible to adapt the method used in \cite{ricardo} based on the solution's level sets and deduce the \emph{appearance of $L^{\infty}$-bounds} for solutions to Boltzmann equation.  The method is reminiscent of the De Giorgi's $L^{p}-L^{\infty}$ argument and has been applied in the study of the homogeneous Landau equation in \cite{ABL,GGL24} and would imply, for the Boltzmann equation, that if a solution $f=f(t,v)$ to \eqref{eq:Landau} satisfies \eqref{eq:ProSer1} and suitable $L^{1}$-moments estimates then for any $t_{\star} \in (0,T)$ there exists $K(t_{\star},T) >0$ such that
\begin{equation*}\label{eq:Linfp-intro}
\sup_{t \in [t_{*},T]}\left\|f(t, \cdot)\right\|_{ {\infty}} \leq K(t_{*},T)\,.
\end{equation*}

In addition to the appearance of pointwise bounds, the relevance of the Theorem \ref{mps_g} lies also on the fact that it is central to deduce stability of weak solutions to the Boltzmann equation \eqref{eq:Boltz}.  This was observed in \cite{fournier} in a result that can be roughly summarized as follows.  Finite energy solutions to \eqref{eq:Boltz} lying in 
$$L^{1}([0,T]\,;\,L^{p}(\R^{d})), \qquad p > \frac{d}{d+\g}$$
are unique, see \cite[Corollary 1.5]{fournier}.  The following theorem shows that the condition $f \in L^{1}([0,T],L^{p}(\R^{d}))$ can be replaced with the Prodi-Serrin criterion \eqref{eq:ProSer1}.  
\begin{theo}\label{theo:unique} Assume $d=3$ and let $h_{\rm in},g_{\rm in} \in \mathcal{Y}_{\rm in}$ be given. Let $h=h(t,v),g=g(t,v)$ two classical solutions to Eq. \eqref{eq:Boltz}  associated with initial data
$h(0)=h_{\rm in},$ $g(0)=g_{\rm in}$  
with 
 $$B(u,\sigma)=|u|^{\g}b(\cos \theta)$$
where $b(\cdot)$ satisfies \eqref{eq:Noyau} and $- \frac{3}{2} <\g+2s  <0$.   
Assume that  {$$h_{\rm in}, g_{\rm in} \in L^2_{k+2|\g|}(\R^{d}) \cap L^{1}_{\ell}(\R^{d})$$
 for some $k > 11 + 4s$ and $\ell > k+|\g|$} and that both solutions satisfy the Prodi-Serrin condition:
\begin{equation}\label{eq:ProSergf}
\langle \cdot \rangle^{|\g|} h \, , \langle \cdot \rangle^{|\g|}g \in L^{r}([0,T]\,;\,L^{q}(\R^{d})) \:\: \text{ with }\:\: \frac{2s}{r}+\frac{d}{q}=2s+d+\g\end{equation} 
where $r \in (1,\infty),$ $q \in \left(\frac{d}{d+\g+2s} , \infty \right).$
Then,  there exists $\bm{C}_{T}(h_{\rm in},g_{\rm in}) >0$ depending on $d, \g, T, k,$, $\|h_{\rm in}\|_{L^{1}_{\ell}}, \|g_{\rm in}\|_{L^{1}_{\ell}}$,  $H(h_{\rm in}), H(g_{\rm in})$, and the Prodi-Serrin norm $\left\|\langle \cdot\rangle^{|\g|}h\right\|_{L^{r}_{t} L^q_v }	, \left\|\langle \cdot\rangle^{|\g|}g\right\|_{L^{r}_{t} L^q_v }$ 
such that
\begin{equation}\label{eq:Stability}
\|h(t)-g(t)\|_{L^{2}_{k}}^{2} \leq \bm{C}_{T}\|h_{\rm in}-g_{\rm in}\|_{L^{2}_{k}}^{2} \qquad \forall t \in [0,T].
\end{equation}
In particular, if $h_{\rm in}=g_{\rm in}$, then $h(t)=g(t)$ for all $t \in [0,T].$\end{theo}  
We stress that the main Theorems \ref{mps_g} and \ref{theo:unique} must be interpreted as practical criteria to get \emph{a priori} estimates for solutions to the  Boltzmann equation in an explicit and quantitative way.  The computations provided are formal, certainty the ones requiring smoothness.  For this reason, we are stating the results in the framework of classical solutions.  The rigorous justification of such estimates is a delicate issue, however, it is likely that at least in some regimes (such as smallness of the initial data) or under reasonable restrictions (such as a Prodi-Serrin like criteria) can be carry on, with no condition, for weak solutions.  Evidence of this can be found for example in \cite{GIS,GGL24}.

We point out that the Prodi-Serrin like criterion of Theorem \ref{theo:unique} and Theorem \ref{mps_g} are the same, even if the former deals with a \emph{weighted} $L^{2}$-norm.  Effort has been made to provide a unique criterion at the price of adding an assumption on the initial data $h_{\rm in},g_{\rm in}$.  We also point out that the stability estimate \eqref{eq:Stability} is likely untrue for \emph{unweighted} norms $L^{2}$. For this reason, the cornerstone of Theorem \ref{theo:unique}'s proof is the propagation of \emph{weighted} $L^{p}$-norm under the criterion \eqref{eq:ProSer1}.

\subsection{Link with the Landau equation and strategy of proof} \label{sec:Landau}
The Boltzmann equation \eqref{eq:Boltz} and the Landau equation \eqref{eq:Landau} have been studied in a parallel fashion, with developments in one helping the theory for the other.  The Landau equation can be seen as a limiting equation derived from the Boltzmann equation in the regime of grazing collision.  In such a limiting procedure inherently the angular singularity $s \to 1$.  Rigorous work of this limiting procedure can be found, for example, in \cite{Al-Vi}.  A recent and interesting new point of view can be found in \cite{YZ23}.

The Landau equation is given by
\begin{equation}\label{eq:Landau}
\partial_{t}f(t,v)=\Q_{L}(f)(v) := {\grad}_v \cdot \int_{\R^{d}} |v-\vet|^{\g} \, \Pi(v-\vet) 
\Big\{ \fet \grad f - f {\grad f}_\ast \Big\}
\, \d\vet\, ,
\end{equation}
with commonly used shorthands $f:= f(v)$, $f_* := f(v_*)$ and where $\Pi(z)=\left(|z|^{2}\mathrm{Id}- z \otimes z\right)$. We refer the reader to the introduction of \cite{ABDL-PS} for an extensive description of the existing results, both conditional and unconditional, relevant for the study of \eqref{eq:Landau}. We mention here that a recent contribution \cite{GS} has shown that strong solution to the \eqref{eq:Landau} do not blow-up.  The special structure of the equation leads the Fisher information functional to decrease along the flow of solutions to \eqref{eq:Landau}, refer to \cite{GGL24} for consequences of this result for weak solutions. It is an open question to see if an analogous result holds for the spatially homogeneous Boltzmann equation considered in the present contribution.

The present contribution's approach is reminiscent of \cite{ABDL-PS} for the Landau equation.  In fact, we point out that the statements in Theorem \ref{mps_g} and \ref{theo:unique} become for $s=1$ exactly the corresponding main results in \cite{ABDL-PS}. An interesting problem is to check rigorously whether the results of this paper recover those of \cite{ABDL-PS} in the grazing collision limit (in particular, when $s \to 1$). 

Let us discuss now similarities and difficulties when implementing the strategy to the Boltzmann case.  We notice that the Landau equation can be written as a nonlinear parabolic equation
\begin{equation}\label{LFD}
\partial_{t}f=\mathrm{Tr}\left(\mathcal{A}[f] D^{2}f\right) +c_{d,\g}\,\bm{c}_{\g}[f]f ,\end{equation}
where $D^{2}f$ is the Hessian matrix of $f$, $\mathcal{A}[f]$ is a suitable positive definite diffusion matrix, $c_{d,\g} >0$ a positive constant, and 
$$\bm{c}_{\g}[f]=\int_{\R^{d}}|v-\vet|^{\g}f(\vet)\d \vet, \qquad \forall v \in \R^{d}.$$
Due to coercive properties of the diffusion matrix $\mathcal{A}[f]$, the Landau equation exhibits a regularisation effect corresponding to a gain of one full derivative with the immediate appearance of the $\dot{\mathbb{H}}^{1}$ norm, up to some algebraic weight. Similarly, the Boltzmann operator $\Q$ provides a regularisation of \textit{fractional order} with immediate appearance of $\dot{\mathbb{H}}^{s}$ norm, with an algebraic weight $\langle \cdot \rangle^{\frac{\gamma}{2}}$ and with $s \in (0,1)$ directly related to the angular scattering singularity $s$ in \eqref{eq:Noyau}.  Such gain of regularity has been first rigorously proven, with appearance of local $\dot{\mathbb{H}}^{s}$-norm, in \cite{ADVW} and latter extended in \cite{amu10} to give a suitable coercivity estimate for $\Q$, see Lemma \ref{lem:coerc}.  Such fact can also be seen clearly thanks to the decomposition, derived in \cite{Silv}, of the collision operator $\Q$ as 
$$\Q(f)=\Q_{1}(f) + \tilde{c}_{d,\g}\ \bm{c}_{\g}[f]\,f, \qquad \Q_{1}(f)=\int_{\R^{d}}\left(f(\vet)-f(v)\right)K_{f}(v,\vet)\d \vet\,,$$
for some suitable kernel $K_{f}(v,\vet)$, depending on $f$ and analogous to $\mathcal{A}[f]$ in \eqref{LFD}, which has a regularisation effect similar to the fractional Laplacian of order $s\in(0,1)$
$$K_{f}(v,\vet) \sim C_{f}(v)\,|v-\vet|^{-d-2s}\,, \qquad v,\vet \in\R^{d}\times\R^{d}.$$
For the well-posedness of \eqref{LFD}, the competition between the helping action of the diffusion associated to $\mathcal{A}[f]$ and the singular drift term $\bm{c}_{\g}[f]$ is a crucial point to be addressed, hence, the conditional criterion given in \cite{ABDL-PS} (which is essentially \eqref{eq:ProSer1} with $s=1$). Similarly, for the Boltzmann equation the key point in to control the singular drift term $\bm{c}_{\g}[f]$ with the helping action of the fractional diffusion term $\Q_{1}(f)$. The quantification of such action is obtained by generalising the so-called $\e$-Poincar\'e inequality, derived in \cite{GG,ABDL-PS}, to address fractional regularisation, see Proposition \ref{prop:ePo}.
 
Roughly speaking, the main steps to prove Theorems \ref{mps_g} and \ref{theo:unique} are the following:
\begin{enumerate}
\item [i)] A control of the singular term $\bm{c}_{\g}[f]$ in terms of suitable higher norm $\|\la \cdot \ra^{|\g|}f\|_{L^{q}}^{r}$ with $q,r$ appearing in \eqref{eq:ProSer1}.
\end{enumerate}
This is done thanks to a new fractional $\e$-Poincar\'e inequality which accounts for the fractional diffusion term $\|\la \cdot\ra^{\frac{\g}{2}}f\|_{\dot{\mathbb{H}}^{s}}$ in Proposition \ref{prop:ePo}. With such a $\e$-Poincar\'e inequality, the proof of Theorem \ref{mps_g} is deduced in a direct fashion from the evolution of the $L^{p}$-norm for solution to \eqref{eq:Boltz}.
\begin{enumerate}
\item [ii)]A generalisation of Theorem \ref{mps_g} addressing weighted $L^{p}$-norm of solutions to \eqref{eq:Boltz} and, therefore, a weighted version of the $\e$-Poincar\'e inequality.
\end{enumerate}
Controlling weighted norms of $\|f(t)\|_{L^{p}_{k}}$ to derive the stability estimate in Theorem \ref{theo:unique} is a non trivial matter, especially if one wishes to keep an ``unweighted'' Prodi-Serrin criterion \eqref{eq:ProSer1}, that is, to provide a unified criterion for the control of any weighted norm $\|\cdot\|_{L^{p}_{k}}$, $k \geq0$.  We address this issue by  imposing an additional assumption only on the initial datum to ensure the propagation or appearance of such norm, see Theorem \ref{mpsweight}.   A new weighted version of the $\e$-Poincar\'e inequality is also presented in the proof of such result, see Proposition \ref{prop:ePo-weight}.  Relative to the Landau equation, the proof is more challenging because of the distribution of weights $\la \cdot\ra^{\frac{k}{2}}$ due to the fully \emph{nonlocal} nature of the operator $\Q_{1}(f)$. Details on this are given in Section \ref{sec:weights}.
\begin{enumerate}
\item [iii)]A careful study of the evolution of the $L^{2}_{k}$ norm of solutions to \eqref{eq:Boltz} and combine it with a Gronwall argument.
\end{enumerate}
This step is delicate and resorts to several fine properties of the collision operator, some of them  scattered in the literature. We use in particular several sharp estimates from \cite{He} and \cite{amu11}. In particular, the nonlocal nature of the collision operator $\Q$ presents challenges that require commutator estimates of the type
$$\la \Q(f, \la \cdot \ra^{ \frac{k}{2} } g), \la \cdot \ra^{ \frac{k}{2} } \psi  \ra_{L^2} - \la \Q(f, g) , \la \cdot \ra^{ k}\psi \ra_{L^{2}} \,$$
where $\la \cdot,\cdot\ra_{L^{2}}$ denote standard the inner product in $L^{2}$, see \cite{AMSY,CG} and \cite{He} for fine versions of such commutator estimates. New commutators estimates valid for very soft potentials and of independent interest are provided along the contribution, see for example Proposition \ref{cor:Comm}.

\subsection{Organization of the paper} Section \ref{sec:prodi} provides a proof for Theorem \ref{mps_g} as well as the fractional $\e$-Poincar\'e inequality. We deal with the extension to these results to $L^{p}$-spaces with \emph{weights} in Section \ref{sec:weights} and provide the main stability result yielding to Theorem \ref{theo:unique} in Section \ref{sec:unique}. We collect in Appendix \ref{app:tech} several estimates related to moments estimates and useful properties of the collision operator $\Q$. We in particular extend a fundamental result of \cite[Prop. 4.1]{AMSY} to very soft-potentials in Appendix \ref{app:AMSY}.

\subsection*{Acknowledgements} The second author acknowledges support from the LabEx CEMPI (ANR-11-LABX-0007) and CPER WaveTech@HdF. B.L. was partially supported by PRIN2022 (project ID: BEMMLZ) Stochastic control and games and the role of information. He also gratefully acknowledges the financial support from the Italian Ministry of Education, University and Research (MIUR), Dipartimenti di Eccellenza grant 2022-2027
as well as the support from the de Castro Statistics Initiative, Collegio Carlo Alberto (Torino). The authors wish to express their  gratitude to the anonymous referees whose comments helped to improve several aspects of the presentation.

 \emph{Data sharing not applicable to this article as no datasets were generated or analysed during the current study.
}
\subsection*{Note added to the paper} \label{sec:GS}

Since the submission of this work for publication, an important contribution \cite{ISV} has appeared as a preprint, demonstrating that the Fisher information is monotone decreasing in time along solutions to \eqref{eq:Boltz}. This finding, which aligns with the results in \cite{GS} concerning the Landau equation, in particular establishes the global existence of smooth solutions to \eqref{eq:Boltz}. Furthermore, it implies that the Prodi-Serrin criteria discussed herein are automatically satisfied if the initial datum fulfills certain regularity conditions. Nonetheless, we believe that the tools and techniques developed in this paper remain of significant interest for further studies of the Boltzmann equation, especially for (very) soft potentials.

\section{Prodi-Serrin like criteria for $L^{p}$-appearance and propagation} \label{sec:prodi}

\subsection{General properties of the Boltzmann operator}

We collect here several fundamental properties of the Boltzmann collision operator that will be used throughout the sequel. We begin with the following \emph{Cancellation Lemma} obtained in \cite[Section 3]{ADVW}.
\begin{lem}[\textit{\textbf{Cancellation Lemma}}]\label{lem:canc} Given a collision kernel $B(u,\sigma)=B(|u|,\cos\theta)$, for almost every $\vet \in \R^{d}$ and any smooth $F$,  one has
\begin{equation}\label{eq:Can}
\int_{\R^{d}}\d v\int_{\S^{d-1}}B(u,\sigma)\,\left[F(v')-F(v)\right]\,\d\sigma=\left(F \ast S_{B}\right)(\vet)
\end{equation}
where
$$S_{B}(z)=|\S^{d-2}|\int_{0}^{\frac{\pi}{2}}\sin^{d-2}(\theta)\left[\frac{1}{\cos^{d}\left(\tfrac{\theta}{2}\right)}B\left(\frac{|z|}{\cos\left(\tfrac{\theta}{2}\right)}\,,\,\cos\theta\right)-B(|z|\,,\,\cos\theta)\right]\d\theta.$$
\end{lem}
\begin{rmq}\label{nb:canc} For $B(|u|,\cos\theta)=|u|^{\g}b(\cos\theta)$, one sees that
$$S_{B}(z)=|z|^{\g}|\S^{d-2}|\int_{0}^{\frac{\pi}{2}}\sin^{d-2}(\theta)\tilde{b}(\cos\theta)\d\theta$$
where
$$\tilde{b}(\cos\theta)=\left[\left(\cos\left(\tfrac{\theta}{2}\right)\right)^{-\g-d}-1\right]b(\cos\theta), \qquad \theta \in \left[0,\frac{\pi}{2} \right].$$
Notice that, under the assumption \eqref{eq:b0s} one has
$$\sin^{d-2}(\theta)\tilde{b}(\cos\theta) \sim \frac{b_{0}(\g+d)}{8}\theta^{1-2s}, \qquad \theta \sim 0\,,$$ 
so that $\tilde{b}$ is integrable  for any $s \in (0,1)$ and 
$$S_{B}(z)=\int_{\S^{d-1}}\tilde{B}(u,\sigma)\d\sigma=|z|^{\g}\int_{\S^{d-1}}\tilde{b}(\widehat{z}\cdot \sigma)\d\sigma=\|\tilde{b}\|_{L^1}|z|^{\g}$$
is well-defined.
\end{rmq}
We also recall the following coercivity estimate related to the collision operator, taken from the proof of \cite[Eq. (2.9)]{amuxy} (we also refer to \cite[Theorem 1]{strain} for an equivalent formulation):
\begin{lem}[\textit{\textbf{Coercivity estimate}}]\label{lem:coerc} Let $$B(u,\sigma)=|u|^{\g}b(\cos\theta)\,,$$ where $b$ satisfies \eqref{eq:b0s} and $\g+2s \leq 0.$ For any  suitable $G \in \mathcal{Y}_{\rm in}$ and $F  \in \mathcal{S}(\R^{d})$, define
\begin{equation}\label{eq:MDG}
\mathscr{D}[G,F]=\int_{\R^{d}}\d v\int_{\R^{d}}\d\vet\int_{\S^{d-1}}B(u,\sigma)G(\vet)\left[F(v')-F(v)\right]^{2}\,\d\sigma.\end{equation}
Then, there exist $c_{0},C_{0} >0$ (depending only on $d,\g,\varrho_{0},E_{0},H_{0})$ such that
$$\mathscr{D}[G,F] \geq c_{0}\left\|\langle \cdot \rangle^{\frac{\g}{2}}F\right\|_{\dot{\mathbb{H}}^{s}}^{2}-C_{0}\,\left\|\langle \cdot\rangle^{\frac{\g}{2}}F\right\|_{L^2}^{2}$$
holds true for any $G \in \mathcal{Y}_{\rm in}$ and $F \in \mathcal{S}(\R^{d})$.
\end{lem}
\begin{rmq}\label{rmq:Dgamma} For $B(u,\sigma)=|u|^{\g}b(\cos\theta)$ satisfying \eqref{eq:b0s}, we simply write $\mathscr{D}[G,F]$ for the aforementioned functional. To enlighten the role of the parameter $\g$ in the above definition, we will also use the notation $\Q_{\g}(g,f)$ for the collision operator associated to $B(u,\sigma)=|u|^{\g}b(\cos\theta)$ and denote with $\mathscr{D}_{\g}[G,F]$ the associated coercivity functional, see, e.g. Proposition \ref{prop:AMSY}.
\end{rmq}
Introduce the notation
\begin{equation}\label{eq:cgamma}
\bm{c}_{\g}[f](v)=\int_{\R^{d}}|v-\vet|^{\g}f(\vet)\d \vet, \qquad \g <0\,.\end{equation}
The analysis of this term requires then the specific use of the Hardy-Littlewood-Sobolev inequality which we recall here for the sake of clarity.
\begin{prop}[\textit{\textbf{Hardy-Littlewood-Sobolev inequality}}] \label{prop:HLS}
Let $d\in \N$, $d\ge 1$, $1 <  m,p < \infty$ and $0 < \lambda < d$ with
$$\frac{1}{p}+\frac{\lambda}{d}+\frac{1}{m}=2.$$
Then, there exists $C_{\rm{HLS}}>0$ (depending on $d,p,\lambda$) such that for c $g,h\::\:\R^{d}\to \R$
\begin{equation}\label{eq:HLS}
\int_{\R^{2d}}g(x)|x-y|^{-\lambda}h(y)\d x \d y \leq C_{\rm{HLS}}\, \|g\|_{L^{p}}\, \|h\|_{L^{m}} .
\end{equation}
\end{prop}
 
We finally recall the \emph{weak form} of the collision operator $\Q(g,f)$ (see \cite{villH}), for any smooth test function $\phi=\phi(v)$
\begin{equation}\label{eq:weak}
\int_{\R^{d}}\Q(g,f)\phi \d v =\int_{\R^{d}\times \R^{d}\times \S^{d-1}}B(v-\vet,\sigma)g(\vet)f(v)\left[\phi(v')-\phi(v)\right]\d v\d\vet\d\sigma\,.\end{equation}

\subsection{Evolution of $L^{p}$-norm}

We investigate in details the evolution of $L^{p}$-norms for solutions to the Boltzmann equation \eqref{eq:Boltz}. 
We use in the sequel the notation, for $k\in \R$ and $p \in (1,\infty)$,
\begin{equation}\label{eq:MkpD}
\lM_{k,p}(t) :=\int_{\R^{d}}f(t,v)^{p}\langle v\rangle^{k}\d v, \qquad \lD_{s,k,p}(t) :=\left\|\langle \cdot \rangle^{\frac{k}{2}}f(t)^{\frac{p}{2}}\right\|_{\dot{\mathbb{H}}^{s}}^{2}\,,
\end{equation}
where $s \in (0,1)$. We adopt also the following short-hand notation
$$\lM_{p}(t)=\lM_{0,p}(t) \qquad \text{ and } \quad \lD_{s,p}(t)=\lD_{s,\g,p}(t).$$
One has then the following proposition.
\begin{prop}\label{prop:dLmp} Let $f_{\rm in} \in \mathcal{Y}_{\rm in}$ be given and let $f(t,\cdot)$ be a classical solution to \eqref{eq:Boltz} and let $p >1$. Then, there exists $C_{p},c_{p} >0$ depending only on $p,\|f_{\rm in}\|_{L^{1}_{2}}$ and $H(f_{\rm in})$ as well as $\bm{C}_{B} >0$ depending on the collision kernel $B$ such that
\begin{equation}\label{eq:dMp}
\frac{\d}{\d t}\lM_{p}(t) + c_{p}\lD_{s,p}(t) \leq (p-1)\bm{C}_{B}\int_{\R^{d}}\bm{c}_{\g}[f(t)](v)f^{p}(t,v)\d v + C_{p}\lM_{\g,p}(t)\,,
\end{equation}
for any $t \geq0.$
\end{prop}
\begin{proof} We multiply \eqref{eq:Boltz} by $f^{p-1}(t,v)$ and integrate over $\R^{d}$ to deduce that
\begin{multline*}
\frac{1}{p}\dfrac{\d}{\d t}\lM_{p}(t)=\int_{\R^{d}}\partial_{t}f(t,v)f^{p-1}(t,v)\d v=\int_{\R^{d}}f^{p-1}(t,v)\Q(f,f)\d v\\
=\int_{\R^{d}\times \R^{d}\times \S^{d-1}} B(u\,,\sigma)f(t,\vet)f(t,v)\left[f(v')^{p-1}-f(v)^{p-1}\right]\d v\d\vet \d\sigma.\end{multline*}
Therefore, applying \eqref{eq:XY} with $X=f(v')$ and $Y=f(v)$ one sees that
$$\dfrac{\d}{\d t}\lM_{p}(t)\leq \mathcal{I}_{1}[f(t)]-\mathcal{I}_{2}[f(t)]\,,$$
with
$$\mathcal{I}_{1}[f(t)]=\frac{p}{p'}\int_{\R^{2d}\times \S^{d-1}}f(t,\vet)\left[f(t,v')^{p}-f(t,v)^{p}\right]B(u,\sigma)\d\vet\d v\d\sigma\,,$$
and 
$$\mathcal{I}_{2}[f(t)]=\frac{p}{\max(p,p')}\int_{\R^{2d}\times\S^{d-1}}f(t,\vet)\left[f(t,v')^{\frac{p}{2}}-f(t,v)^{\frac{p}{2}}\right]^{2}B(u,\sigma)\d\vet\d v\d\sigma.$$
According to the Cancellation Lemma \ref{lem:canc}, and with the notations of  Remark \ref{nb:canc},
$$\mathcal{I}_{1}[f(t)]=\frac{p}{p'}\int_{\R^{2d}\times \S^{d-1}}f(t,\vet)f^{p}(t,v)\tilde{B}(u,\sigma)\d\vet\d v\d\sigma, \qquad \tilde{B}(u,\sigma)=|u|^{\g}\tilde{b}(\cos \theta),$$
and, using again Remark \ref{nb:canc},
$$\mathcal{I}_{1}[f(t)] \leq \frac{p}{p'}\|\tilde{b}\|_{L^{1}(\S^{d-1})}\int_{\R^{2d}}f(t,\vet)f^{p}(t,v)|v-\vet|^{\g}\d v\d\vet.$$
Moreover, applying the coercivity estimate in Lemma \ref{lem:coerc} with $G=f(t,\cdot) \in \mathcal{Y}_{\rm in}$ and $F=f^{\frac{p}{2}}(t,\cdot)$ shows that there exists $c_{p},C_{p} >0$ depending only on $\|f_{\rm in}\|_{L^{1}_{2}}$, $H(f_{\rm in})$ and $p$ such that
$$\mathcal{I}_{2}[f(t)] \geq c_{p}\lD_{s,p}(t)  - C_{p}\lM_{\g,p}(t).$$
Putting these estimates together gives the result. 
\end{proof}
One sees from the previous estimates that, to study the evolution and appearance of $L^{p}$-norms, we need to understand the contribution of the term
$$\int_{\R^{d}}\bm{c}_{\g}[f(t)](v)f^{p}(t,v)\d v.$$
This is done by the following fractional $\e$-Poincar\'e, inspired by the work \cite{GG}.  The following proof is extracted from \cite{ABDL-PS} where the result is proven for $s=1$.  Refer also to Proposition \ref{prop:ePo-weight} for a generalisation including suitable weights.
\begin{prop}[\textit{\textbf{Fractional $\e$-Poincar\'e}}]\label{prop:ePo} Assume that $d\in\N$, $d\ge 2$, and consider $s \in (0,1]$  and $-d < \g+2s \leq 0$, and
$$\frac{d}{d+2s+\g} < q  < \frac{d}{d+\g}. $$
Then, there exists $C_{0} >0$ depending only on $d, \g, s,q$ such that, for any $\e>0$ and suitable functions $\phi$ and $g\ge0$,
\begin{equation}\label{eq:estimatc}
\int_{\R^{d}}\phi^{2}\bm{c}_{\g}[g]\d v \leq \e\,\left\|\langle \cdot \rangle^{\frac{\g}{2}}\phi\right\|_{\dot{\mathbb{H}}^{s}}^{2}
+C_{0}\left( \|g\|_{L^1}  +\e^{-\frac{\nu}{s-\nu}}\left\|\langle \cdot \rangle^{|\g|}g\right\|_{L^q}^{\frac{s}{s-\nu}}\right)\int_{\R^{d}}\phi^{2}\langle v\rangle^{\g}\d v,
\end{equation}
where $\nu \in (0,s)$ is given by $\nu=\frac{d-q(d+\g)}{2q}.$
\end{prop} 
\begin{proof} 
 For a given  $g,\phi\ge 0$, we define
$$
\mathcal{A}[g,\phi] :=\int_{\R^{d}}\phi^{2}(v)\bm{c}_{\g}[g](v)\d v=\int_{\R^{d}\times \R^{d}}|v-\vet|^{\g}\phi^{2}(v)g(\vet)\d v\d\vet\,.
$$
For any $v,\vet \in \R^{d}$, if $|v-\vet| < \frac{1}{2}\langle v\rangle$, then $\langle v\rangle\leq 2\langle \vet\rangle$.  We deduce from this, see \cite[Eq. (2.5)]{amuxy}, that
$$|v-\vet|^{\g}\leq 2^{-\g}\langle v\rangle^{\g}\left(\ind_{\left\{|v-\vet| \geq\frac{\langle v\rangle}{2}\right\}}+
\langle \vet\rangle^{-\g}|v-\vet|^{\g}\ind_{\left\{|v-\vet|< \frac{\langle v\rangle}{2}\right\}}\right).$$
Therefore,
\begin{equation}\label{eq:ineq}
\mathcal{A}[g,\phi] \leq  2^{-\g}\left(\mathcal{A}_{1}[g,\phi]+\mathcal{A}_{2}[g,\phi]\right),
\end{equation}
where $\mathcal{A}_{1}[g,\phi],\mathcal{A}_{2}[g,\phi]$ denote respectively the ``regular'' and ``singular'' parts of $\mathcal{A}[g,\phi]$,
\begin{multline*}
\mathcal{A}_{1}[g,\phi]:=\int_{\R^{d}}\langle v\rangle^{\g}\phi^{2}(v)\d v \int_{|v-\vet| \geq \frac{\langle v\rangle}{2}} g(\vet)\d\vet\,,\\
\text{and}\qquad \mathcal{A}_{2}[g,\phi]:=\int_{\R^{d}}\langle \vet\rangle^{-\g}g(\vet)\d \vet\int_{|v-\vet| < \frac{1}{2}\langle v\rangle }|v-\vet|^{\g}\langle v\rangle^{\g}\phi^{2}(v)\d v.\end{multline*}
Introduce $G(v)=\langle v\rangle^{|\g|}g(v)$ and $\psi(v)=\langle v\rangle^{\frac{\g}{2}}\phi(v)\,$.  Clearly,
$$\mathcal{A}_{1}[g,\phi]  \leq \, \| {g}\|_{L^{1}} \|\psi\|_{L^{2}}^{2}\,,$$
while
$$\mathcal{A}_{2}[g,\phi] \leq \int_{\R^{d}\times \R^{d}}|v-\vet|^{\g}G(\vet)\psi^{2}(v)\d v\d\vet\,.$$
We estimate then $\mathcal{A}_{2}[g,\phi]$  thanks to  {Hardy-Littlewood-Sobolev} inequality \eqref{eq:HLS} to get
$$\mathcal{A}_{2}[g,\phi]  \leq C_{\mathrm{HLS}}\,\|G\|_{L^q}\,\|\psi^{2}\|_{L^{m}}=C_{\mathrm{HLS}}\,\|G\|_{L^{q}}\,\|\psi\|_{L^{2m}}^{2}, \qquad \frac{1}{m} =2-\frac{1}{q} {-\frac{|\g|}{d}}\quad  (1< q,m < \infty), $$
where $C_{\mathrm{HLS}}$ depends only on $\g, d, q$.   We apply this with $\frac{d}{d+2s+\g} < q< \frac{d}{d+\g}$ and $s \in (0,1]$, and thus write for some $\nu \in (0,s)$
\begin{equation}\label{eq:prHLS}
q=\frac{d}{d+\g+2\nu}, \qquad m=\frac{d}{d-2\nu}.\end{equation}
Notice that $2m=\frac{2d}{d-2\nu}$ and according for example to Theorem 1.38 of \cite{chemin}, $\dot{\mathbb{H}}^{\nu}$ is continuously embedded in $L^{2m}(\R^{d})$, that is, for a constant $C>0$ depending on $d,\nu,\g,q$,
$$\mathcal{A}_{2}[g,\phi] \leq C \|G\|_{L^q}\|\psi\|_{\dot{\mathbb{H}}^{\nu}}^{2}.$$
Moreover, since 
$$\|\psi\|_{\dot{\mathbb{H}}^{\nu}} \leq \|\psi\|_{\dot{\mathbb{H}}^{s}}^{\frac{\nu}{s}} \,\|\psi\|_{L^2}^{\frac{s-\nu}{s}}\,, $$
see for example \cite[Proposition 1.32]{chemin}, one has that
$$\mathcal{A}_{2}[g,\phi] \leq C \|G\|_{L^q} \|\psi\|_{\dot{\mathbb{H}}^{s}}^{\frac{2\nu}{s}}\,\|\psi\|_{L^2}^{\frac{2s-2\nu}{s}}.$$
Thanks to Young's inequality, given $\delta >0$, one deduces that
\begin{equation}\label{eq:A2C}
2^{-\g}\mathcal{A}_{2}[g,\phi] \leq \delta\,\|\psi\|_{\dot{\mathbb{H}}^{s}}^{2}+C^{\frac{s}{s-\nu}}\|G\|_{L^q}^{\frac{s}{s-\nu}}\,\delta^{-\frac{\nu}{s-\nu}}\|\psi\|_{L^2}^{2}\,.\end{equation}
Plugging this inequality into the estimate for $\mathcal{A}_{1}[g,\phi]$ we see that 
$$\mathcal{A}[g,\phi] \leq \delta\,\|\psi\|_{\dot{\mathbb{H}}^{s}}^{2} + C_{0}\left(\|g\|_{L^1} +\delta^{-\frac{\nu}{s-\nu}}\|G\|_{L^{q}}^{\frac{s}{s-\nu}}\right)\|\psi\|_{L^2}^{2},$$
for some positive constant $C_{0}$ depending on $s,\nu,\g,q$. This proves the result recalling the definition of $G$ and $\psi$. 
\end{proof}

\subsection{Appearance and propagation of $L^{p}$-norms}

With the estimates established in the previous sections, we can prove the main result about appearance and propagation of $L^{p}$-norms for solutions to the Boltzmann equation \eqref{eq:Boltz}, providing the proof of Theorem \ref{mps_g} in the Introduction.
 
\begin{proof}[Proof of Theorem \ref{mps_g}] We assume $f$ satisfies the Prodi-Serrin criterion \eqref{eq:ProSer1} with some fixed $q > \max\left(1,\frac{d}{d+\g+2s}\right).$ We write, as in \eqref{eq:prHLS}, $q=\frac{d}{d+\g+2\nu}$, $\nu \in [0,s]$ so that $r=\frac{s}{s-\nu}$ satisfies
$$\frac{d}{q}+\frac{2s}{r}=d+\g+2s.$$
Starting from \eqref{eq:dMp} and applying the above fractional $\e$-Poincar\'e inequality with $q$ as in  \eqref{eq:ProSer1}, $g=f(t,v)$ and $\phi=f^{\frac{p}{2}}(t,v)$, we deduce that, for any $\e >0$,
\begin{multline*}
\frac{\d}{\d t}\lM_{p}(t) + c_{p}\lD_{s,p}(t) 
\leq \e(p-1)\bm{C}_{B} \lD_{s,p}(t) \\
+ \left((p-1)\bm{C}_{B}C_{0}\left[\|f(t)\|_{L^1}+\e^{-\frac{\nu}{s-\nu}}\|\langle \cdot \rangle^{|\g|}f(t)\|_{L^q}^{\frac{s}{s-\nu}}\right]+C_{p}\right)\lM_{\g,p}(t)\,,\end{multline*}
with $\nu=\frac{d-q(d+\g)}{2q}.$ Choosing then $\e$ so that $\e(p-1)\bm{C}_{B}=\frac{c_{p}}{2}$ and noticing that $\|f(t)\|_{L^1}=\|f_{\rm in}\|_{L^1}$, we deduce that there exists $\bm{C}_{1}$ depending on $B$, $p,q,d$, $H(f_{\rm in})$ and $\|f_{\rm in}\|_{L^{1}_{2}}$ such that
\begin{equation}\label{eq:MainLp}\frac{\d}{\d t}\lM_{p}(t) + \frac{c_{p}}{2}\lD_{s,p}(t) \leq \bm{C}_{1}\left(1+\|\langle \cdot \rangle^{|\g|}f(t)\|_{L^q}^{\frac{s}{s-\nu}}\right)\lM_{\g,p}(t).\end{equation}
Recalling $r=\frac{s}{s-\nu}$, we deduce that  the mapping
$$\bm{\Lambda}(t)= \bm{C}_{1}\left(1+\|\langle \cdot \rangle^{|\g|}f(t)\|_{L^q}^{\frac{s}{s-\nu}}\right)=\bm{C}_{1}\left(1+\|\langle \cdot \rangle^{|\g|}f(t)\|_{L^q}^{r}\right) \in L^{1}([0,T])$$
falls under the Prodi-Serrin condition \eqref{eq:ProSer1}. Since $\lM_{\g,p}(t) \leq \lM_{p}(t)$, the inequality \eqref{eq:MainLp} implies that
$$\dfrac{\d}{\d t}\lM_{p}(t) \leq \bm{\Lambda}(t)\lM_{p}(t), \qquad \lM_{p}(0)=\|f_{\rm in}\|_{L^p}^{p} < \infty\,,$$
where we recall that $f_{\rm in} \in L^{p}(\R^{d})$. Then, Gronwall lemma implies that
$$\sup_{t\in [0,T]}\lM_{p}(t) \leq \exp\left(\int_{0}^{T}\bm{\Lambda}(t)\d t\right)\lM_{p}(0)\,,$$
which is nothing but \eqref{eq:propaLp1}.  In order to show the appearance of $L^{p}$-norms, we need to exploit the coercive term in the right-hand side of \eqref{eq:MainLp}. We begin with recalling the Sobolev embedding of $\dot{\mathbb{H}}^{s}(\R^{d}) \hookrightarrow L^{\frac{2d}{d-2s}}(\R^{d})$ with sharp constant $C_{\mathrm{Sob},s}$
$$\|g\|_{L^{\frac{2d}{d-2s}}} \leq C_{\mathrm{Sob},s}\|g\|_{\dot{\mathbb{H}}^{s}}, \qquad \forall g \in \dot{\mathbb{H}}^{s}(\R^{d})\,,$$
which implies
\begin{equation}\label{eq:sobo-p}
\left\|\langle \cdot\rangle^{\frac{\g}{p}}f(t)\right\|_{L^{\frac{pd}{d-2s}}}^{p}=\left\|\langle \cdot\rangle^{\frac{\g}{2}}f^{\frac{p}{2}}(t)\right\|_{L^{\frac{2d}{d-2s}}}^{2} \leq C_{\mathrm{Sob},s}^{2}\, \left\|\langle \cdot \rangle^{\frac{\g}{2}}f^{\frac{p}{2}}(t)\right\|_{\dot{\mathbb{H}}^{s}}^{2}\,. \end{equation}
As in \cite{ABDL-PS}, we resort to standard H\"older interpolation inequality with weights, 
\begin{equation}\label{int-ineqp_bis}
\|\langle \cdot \rangle^{a_{0}}g\|_{L^{r_{0}}} \leq \|\langle \cdot \rangle^{a_{1}}g\|_{L^{r_{1}}}^{\theta}\,\|\langle \cdot \rangle^{a_{2}}g\|_{L^{r_{2}}}^{1-\theta}\,,
\end{equation}
with $r_{0},r_1,r_2 \ge1$, $a_{0},a_1, a_2 \in \R$,
$$\frac{1}{r_{0}}=\frac{\theta}{r_{1}}+\frac{1-\theta}{r_{2}}, \qquad a_{0}=\theta\,a_{1}+(1-\theta)a_{2},  \qquad \theta \in (0,1).$$
Applying such an inequality with the choice $g=f(t)$ and 
$$ a_{0} :=0, \:\:a_1 :=\eta_{p},\:\: a_2 :=\frac{\g}{p}, \qquad r_{0} :=p,\:\: r_1 :=1, \:\:r_2 :=\frac{dp}{d-2s}, \qquad \theta := \frac{2s}{d(p-1)+2s}, $$ 
we get, with \eqref{eq:sobo-p}, that
\begin{equation*} 
\lM_{p}(t) \leq \lm_{\eta_{p}}(t)^{p\theta}\,\left\|\langle \cdot \rangle^{\frac{\g}{p}}f(t,\cdot)\right\|_{L^{\frac{dp}{d-2s}}}^{p(1-\theta)}\leq \lm_{\eta_{p}}(t)^{p\theta}C_{\mathrm{Sob},s}^{2-2\theta}\,\left\|\langle \cdot \rangle^{\frac{\g}{2}}f^{\frac{p}{2}}(t)\right\|_{\dot{\mathbb{H}}^{s}}^{2(1-\theta)},
 \end{equation*}
 holds for any $t\geq0.$ Reformulating this inequality as
 $$\lD_{s,p}(t) \geq C_{\mathrm{Sob},s}^{-2}\lm_{\eta_{p}}(t)^{-\frac{2sp}{d(p-1)}}\lM_{p}(t)^{1+\frac{2s}{d(p-1)}}\,,$$
 and plugging this into \eqref{eq:MainLp} we deduce that
$$\dfrac{\d}{\d t}\lM_{p}(t) + \frac{c_{p}}{C_{\mathrm{Sob},s}^{2}}\lm_{\eta_{p}}(t)^{-\frac{2sp}{d(p-1)}}\lM_{p}(t)^{1+\frac{2s}{d(p-1)}} \leq \bm{\Lambda}(t)\lM_{p}(t).$$
Defining $y_{p}(t)=\lM_{p}(t)\exp\left(-\int_{0}^{t}\bm{\Lambda}(\tau)\d\tau\right)$, one sees then that
$$\dfrac{\d}{\d t}y_{p}(t) \leq -\frac{c_{p}}{C_{\mathrm{Sob},s}^{2}}\lm_{\eta_{p}}(t)^{-\frac{2sp}{d(p-1)}}y_{p}(t)^{1+\frac{2s}{d(p-1)}}$$
which, after integration, yields  
$$y_{p}(t) \leq \left[\frac{2sc_{p}}{d(p-1)C_{\mathrm{Sob},s}^{2}}\left(\sup_{\tau\in [0,T]}\lm_{\eta_{p}}(t)\right)^{-\frac{2sp}{d(p-1)}}\,t\right]^{-\frac{d(p-1)}{2s}}.$$
Recalling that $\|f(t)\|_{L^p}=y_{p}(t)^{\frac{1}{p}}\exp\left(\dfrac{1}{p}\ds\int_{0}^{t}\bm{\Lambda}(\tau)\d\tau\right)$ and setting 
$$\bm{K}_{p,B,q,T}=\left[\frac{2sc_{p}}{d(p-1)C_{\mathrm{Sob},s}^{2}}\right]^{-\frac{d(p-1)}{2sp}}\exp\left(\frac{\bm{C}_{1}}{p}\int_{0}^{t}\left(1+\|\langle \cdot\rangle^{|\g|}f(t)\|_{L^q}^{r}\right)\d t\right)$$
gives the result.
\end{proof}
 
\subsection{Endpoint estimate for $r=1$}
In the aforementioned result the case $r=1$ has been excluded.  We can provide a criterion similar to \eqref{eq:ProSer1} in this case. Notice that $r=1$ corresponds to $q=\frac{d}{d+\g}$ and, for such a result, no additional moment $\la \cdot \ra^{|\g|}$ is needed in \eqref{eq:ProSer1-1}.  However, the criterion only covers the range $1 < p < \frac{d}{d+\g}$.
\begin{prop}
Let $f_{\rm in} \in \mathcal{Y}_{\rm in}$ be given and let $f=f(t,v)$ be a classical solution to Eq. \eqref{eq:Boltz} with 
$$B(u,\sigma)=|u|^{\g}b(\cos \theta)\,,$$
where $b(\cdot)$ satisfies  \eqref{eq:b0s}--\eqref{eq:g2s}.  Assume that the solution $f$ satisfies 
\begin{equation}\label{eq:ProSer1-1}
\,f \in L^{1}\big([0,T]\,;\,L^{\frac{d}{d+\g}}(\R^{d})\big)\,.\end{equation} 
Then, for any $p \in (1,\frac{d}{d+\g})$ the following integrability properties hold:
\begin{enumerate}
\item [(a)]If $f_{\rm in} \in L^{p}(\R^{d})$ then
\begin{equation}\label{eq:propaL1}
\sup_{t \in [0,T]}\|f(t)\|_{L^p} \leq \exp\left(\frac{\bm{C}_{0}}{p}\int_{0}^{T} \|f(t)\|_{L^{\frac{d}{d+\g}}} \d t\right)\|f_{\rm in}\|_{L^p}\,,
\end{equation}
for some explicit constant $\bm{C}_{0}$ depending on $p,d,B$ and $\varrho_{\rm in}$, $E_{\rm in}$, $H_{\rm in}$ but not on $T$ or $\|f_{\rm in}\|_{L^p}$.

\item  [(b)]If $f_{\rm in} \in L^{1}_{\eta_{p}}(\R^{d})$ with $\eta_{p}$ defined in Theorem \ref{mps_g} ($1 < p < \frac{d}{d+\g}$), then, the solution $f=f(t,v)$ satisfies  the following estimate for $t \in (0,T]$
	\begin{equation}\label{Lp-Kpq}
		\|f(t)\|_{L^p} \le \bm{K}_{p,B,T} \,\, t^{- \frac{d}{2s} \left(1 - \frac1{p} \right)}\sup_{\tau \in [0,T]}\lm_{\eta_{p}}(\tau),\end{equation}
	where $\bm{K}_{p,B,T}$ is explicitly depending on $\ds\int_{0}^{T} \|f(t)\|_{L^{\frac{d}{d+\g}} }\d t$ and also on $p$, $B$, $d$ and $\varrho_{\rm in}$, $E_{\rm in}$, $H_{\rm in}$.
\end{enumerate}
\end{prop}
\begin{proof} Starting from Proposition \ref{prop:dLmp} one observes from \eqref{eq:dMp} that only the term
$$\int_{\R^{d}}\bm{c}_{\g}[f(t)](v)f^{p}(t,v)\d v$$
needs to be estimated.  Notice that under \eqref{eq:ProSer1-1}, the $\e$-Poincar\'e inequality and the regularisation effect induced are not needed. Indeed, as in \cite{ABDL-PS}, one can deduce directly from the Hardy-Littlewood-Sobolev inequality \eqref{eq:HLS} that
$$ \int_{\R^{d}}\bm{c}_{\g}[f(t)](v)f^{p}(t,v)\d v\ \leq C_{d,\g,p}\|f(t)\|_{L^{q_{0}}}\|f(t)\|_{L^{p}}, \quad q_{0}=\frac{d}{d+\g}, \quad 1 < p < \frac{d}{d+\g}\,,$$
for some constant depending only on $d,\g$ and $p$.  This and \eqref{eq:dMp} give that
$$\frac{\d}{\d t}\lM_{p}(t) + c_{p}\lD_{s,p}(t) \leq \Lambda_{0}(t)\lM_{p}(t)\,,$$
with 
$$\Lambda_{0}(t)=C_{d,\g,p}(p-1)\bm{C}_{B}\|f(t)\|_{L^{q_{0}}}+C_{p} \in L^{1}([0,T])$$
according to \eqref{eq:ProSer1-1}. This proves the propagation of the $L^{p}$-norms for $1 < p < \frac{d}{d+\g}$. The proof of the appearance follows then the lines of the Theorem \ref{mps_g}.
\end{proof}
\begin{rmq} It is possible, as in \cite[Proposition ]{ABDL-PS} to derive a criterion in the endpoint estimate $r=\infty$, $q=\frac{d}{d+\g+2s}$ as well. Such a case requires a smallness assumption on the norm
$$\sup_{t \in [0,T]}\|\la \cdot \ra^{|\g|}f(t)\|_{L^{\frac{d}{d+\g+2s}}}.$$
Details are left to the reader.
\end{rmq}
\section{Adding moments}\label{sec:weights}
We focus in this section on the evolution of weighted-$L^{p}$ norms of $f(t,v)$ as defined in \eqref{eq:MkpD} adapting the strategy of the previous section in order to incorporate weights.
\subsection{Evolution of weighted $L^{p}$-norms}  Let $f_{\rm in} \in \mathcal{Y}_{\rm in}$ be given and let $f=f(t,v)$ be a classical solution  to Eq. \eqref{eq:Boltz} with 
 $$B(u,\sigma)=|u|^{\g}b(\cos \theta)$$
where $b(\cdot)$ satisfies \eqref{eq:b0s} and $\g+2s <0.$
We recall the notation \eqref{eq:MkpD} and compute the evolution of $\lM_{k,p}(t)$, for $k \geq0$. To do so, we introduce the additional notation
$$F(t,v)=\langle v\rangle^{\frac{k}{p}}f(t,v)\,,$$
so that $\lM_{k,p}(t)=\|F(t)\|_{L^p}^{p}.$ One has that
$$\dfrac{1}{p}\dfrac{\d}{\d t}\lM_{k,p}(t)=\frac{1}{p}\dfrac{\d}{\d t}\|F(t)\|_{L^p}^{p}=\int_{\R^{d}}\Q(f,f)F^{p-1}(t,v) \la v \ra^{ \frac{k}{p} } \d v,$$
that is,
$$\dfrac{1}{p}\dfrac{\d}{\d t}\lM_{k,p}(t)=\int_{\R^{2d}\times \S^{d-1}}f(t,\vet)f(t,v)\left[F^{p-1}(t,v')\langle v'\rangle^{\frac{k}{p}}-F^{p-1}(t,v)\langle v\rangle^{\frac{k}{p}}\right]B(u,\sigma)\d \sigma\d\vet\d v.$$
We follow \cite[Section 2.2]{ricardo} to compute $f(t,v)\left[F^{p-1}(t,v')\langle v'\rangle^{\frac{k}{p}}-F^{p-1}(t,v)\langle v\rangle^{\frac{k}{p}}\right]$.  Write
$$\frac{1}{p}\dfrac{\d}{\d t}\lM_{k,p}(t)=\mathscr{J}_{1}[f(t)]+\mathscr{J}_{2}[f(t)]+\mathscr{J}_{3}[f(t)]\,,$$
with
\begin{equation*}
\mathscr{J}_{1}[f(t)]=\int_{\R^{2d}\times\S^{d-1}}f(t,\vet)F(t,v)\left[F^{p-1}(t,v')-F^{p-1}(t,v)\right]B(u,\sigma)\d\sigma\d\vet\d v\,,
\end{equation*}
\begin{equation*}
\mathscr{J}_{2}[f(t)]=\int_{\R^{2d}\times\S^{d-1}}f(t,\vet)f(t,v)F^{p-1}(t,v)\left[\langle v'\rangle^{\frac{k}{p}}-\langle v\rangle^{\frac{k}{p}}\right]B(u,\sigma)\d\sigma\d\vet\d v\,,
\end{equation*}
and
\begin{equation*}\begin{split}
\mathscr{J}_{3}[f(t)]=\int_{\R^{2d}\times\S^{d-1}}f(t,\vet)f(t,v)&\left[F^{p-1}(t,v')-F^{p-1}(t,v)\right]\\
&\phantom{++}\times \left[\langle v'\rangle^{\frac{k}{p}}-\langle v\rangle^{\frac{k}{p}}\right]B(u,\sigma)\d\sigma\d\vet\d v\,.\end{split}
\end{equation*}
The estimates of the various terms $\mathscr{J}_{i}[f(t)]$ are given in a series of lemmata.
\begin{lem}\label{lem:J1} Let $f_{\rm in} \in \mathcal{Y}_{\rm in}$ be given and let $f(t,\cdot)$ be a classical solution  to \eqref{eq:Boltz} associated to 
$$B(u,\sigma)=|u|^{\g}b(\cos \theta)\,,$$
where $b(\cdot)$ satisfies \eqref{eq:b0s} and $\g+2s <0.$  Given $p >1$, one has that
\begin{equation}\label{eq:J1}
\mathscr{J}_{1}[f(t)] + \frac{1}{\max(p,p')}\mathscr{D}\left[f(t),F^{\frac{p}{2}}(t)\right] \leq  \frac{ {\|\tilde{b}\|_{L^{1}(\S^{d-1})}}}{p'}\int_{\R^{d}}\bm{c}_{\g}[f(t)]F^{p}(t,v)\d v\,,\end{equation}
where $\tilde{b}(\cdot)$ has been defined in Remark \ref{nb:canc} and $\mathscr{D}$ is defined in \eqref{eq:MDG}.
\end{lem}
\begin{proof} The proof follows exactly the lines of  the proof of Proposition \ref{prop:dLmp}.  Notice that
$$
\mathscr{J}_{1}[f(t)] \leq \mathscr{J}_{1,1}[f(t)]-\mathscr{J}_{1,2}[f(t)]\,,$$
with
\begin{equation*}\begin{split}
\mathscr{J}_{1,1}[f(t)]&=
\frac{1}{p'}\int_{\R^{2d}\times\S^{d-1}}f(t,\vet)\left[F^{p}(t,v')-F^{p}(t,v)\right]B(u,\sigma)\d\sigma\d\vet\d v, \\
\mathscr{J}_{1,2}[f(t)]&= \frac{1}{\max(p,p')}\int_{\R^{2d}\times\S^{d-1}}f(t,\vet)
\left[F^{\frac{p}{2}}(t,v')-F^{\frac{p}{2}}(t,v)\right]^{2}B(u,\sigma)\d\sigma\d\vet\d v.
\end{split}\end{equation*}
The computations in Proposition \ref{prop:dLmp} shows that
$$\mathscr{J}_{1,1}[f(t)] \leq \frac{\|\tilde{b}\|_{L^{1}(\S^{d-1})}}{p'}\int_{\R^{d}}\bm{c}_{\g}[f(t)]F^{p}(t,v)\d v\,,$$
while $\mathscr{J}_{1,2}[f(t)]=\frac{1}{\max(p,p')}\mathscr{D} \left[f(t),F^{\frac{p}{2}}(t)\right]$.\end{proof}
For the second term $\mathscr{J}_{2}[f(t)]$ one has the following lemma.
\begin{lem}\label{lem:J2} Let $f_{\rm in} \in \mathcal{Y}_{\rm in}$ be given and let $f(t,\cdot)$ be a  classical  solution to \eqref{eq:Boltz} associated to 
$$B(u,\sigma)=|u|^{\g}b(\cos \theta)\,,$$
where $b(\cdot)$ satisfies \eqref{eq:b0s} and $\g+2s <0.$  Given $p >1$ and $k \geq 2p$ there exists $c_{k,p}(B) >0$ depending only on $p,k$ and on the collision kernel $B$ such that
\begin{equation}\label{eq:J2}
\mathscr{J}_{2}[f(t)] \leq c_{k,p}(B)  \int_{\R^{d}}\bm{c}_{\g+2s}\left[\langle \cdot\rangle^{ \frac{k}{p}-2s} f(t)\right](v) F^{p}(t,v) d v\,,\end{equation}
where we recall that $\bm{c}_{\g+2s}$ is defined through \eqref{eq:cgamma}.
\end{lem}
\begin{proof} Since 
$$\int_{0}^{\pi}b(\cos\theta)\sin^{d}\theta\d\theta = C_{b} < \infty\,,$$
under assumption \eqref{eq:b0s}, one can use \eqref{eq:cancel-int} with $\ell=\frac{k}{p} \ge 2$ and $\alpha=s$ to deduce that there exists $C >0$ depending on $k,p,s,b$ such that 
$$\int_{\S^{d-1}}b(\cos\theta)\left(\langle v'\rangle^{\frac{k}{p}}-\langle v\rangle^{\frac{k}{p}}\right)\d \sigma \leq C|v-\vet|^{2s} \left(\langle v\rangle^{\frac{k}{p}-2s}+\langle  \vet\rangle^{\frac{k}{p}-2s}\right).$$
Therefore,
\begin{equation*}\begin{split}
\mathscr{J}_{2}[f(t)]&=\int_{\R^{2d}\times\S^{d-1}}f(t,\vet)f(t,v)F^{p-1}(t,v)\left[\langle v'\rangle^{\frac{k}{p}}-\langle v\rangle^{\frac{k}{p}}\right]B(u,\sigma)\d\sigma\d\vet\d v\\
&\leq C\int_{\R^{2d}}|v-\vet|^{\g+2s}f(t,\vet) f(t,v)F^{p-1}(t,v)\langle v\rangle^{\frac{k}{p}-2s}\d v\d\vet\\
&\phantom{+++ } +C\int_{\R^{2d}}|v-\vet|^{\g+2s}\langle \vet\rangle^{\frac{k}{p}-2s}f(t,\vet)f(t,v)F^{p-1}(t,v)\d v\d\vet.\end{split}
\end{equation*}
Observing that $\langle \cdot \rangle^{\frac{k}{p}}f F^{p-1}=F^{p}$ we get that
$$\mathscr{J}_{2}[f(t)] \leq C\max_{\beta=2s, \frac{k}{p}  } \int_{\R^{d}}\bm{c}_{\g+2s}\left[\langle \cdot\rangle^{\beta-2s} f(t)\right](v)\left(\langle v\rangle^{-\beta}F^{p}(t,v)\right)\d v\,,$$
which gives the result.\end{proof}
The most delicate term is $\mathscr{J}_{3}[f(t)]$ addressed in the following lemma.
\begin{lem}\label{lem:J3} Let $f_{\rm in} \in \mathcal{Y}_{\rm in}$ be given and let $f(t,\cdot)$ be a classical solution  to \eqref{eq:Boltz} associated to 
$$B(u,\sigma)=|u|^{\g}b(\cos \theta)\,,$$
where $b(\cdot)$ satisfies \eqref{eq:b0s} and $\g+2s <0.$ Let $p >1$ and $k \geq 2p$ be given. For any $\delta >0$ there exists  $C_{\delta}(B) >0$ depending only on $p,k$ and on the collision kernel $B$ such that
\begin{equation}\label{eq:J3}
\mathscr{J}_{3}[f(t)] \leq \delta\mathscr{D}\left[f(t),F^{\frac{p}{2}}(t)\right] 
+C_{\delta}(B)\int_{\R^{d}} F^{p}(t,v)\bm{c}_{\g+2s}\left[\langle \cdot \rangle^{2\frac{k}{p}-2s}f(t) \right](v) \d v\,,\end{equation}
where $\mathscr{D}[\cdot,\cdot]$ is defined in Lemma \ref{lem:coerc}. 
\end{lem}

\begin{proof} Omit here the dependence of $t>0$ for clarity and recall that
\begin{equation*}
\mathscr{J}_{3}[f]=\int_{\R^{2d}\times\S^{d-1}}|v-\vet|^{\g}f(\vet)f(v)\left[F^{p-1}(v')-F^{p-1}(v)\right]\left[\langle v'\rangle^{\frac{k}{p}}-\langle v\rangle^{\frac{k}{p}}\right]b_{s}(\cos\theta)\d\sigma\d\vet\d v\,
\end{equation*}
where $F=\langle \cdot \rangle^{\frac{k}{p}}f$ and write $b_{s}$ instead of $b$ to emphasise the strength of the singularity in \eqref{eq:b0s}. Estimating the difference of weights thanks to \eqref{eq:weight_cancellation}, with $\ell=\frac{k}{p}$ and $\alpha=s \in (0,1)$, induces the control
$$\mathscr{J}_{3}[f] \leq C_{\frac{k}{p},s}\left(\mathscr{J}_{3,1}[f]+\mathscr{J}_{3,2}[f]\right)\,,$$
where, for $\tilde{b}(\cos \theta)=  \sin\left( \tfrac{\theta}{2} \right)b(\cos\theta)$,
$$\mathscr{J}_{3,1}[f]:=\int_{\R^{2d}\times\S^{d-1}} |v - \vet |^{\gamma+s}\tilde{b}(\cos \theta) f( \vet) \la v\ra^{-s}F( v) \left| F^{p-1}( v') - F^{p-1}(v) \right| \d \sigma\d\vet\d v\,,$$
and where the most delicate term to manage is
\begin{equation*}\begin{split}
\mathscr{J}_{3,2}[f]&:= \int_{\R^{2d}\times\S^{d-1}} |v - \vet|^{\gamma+s}\tilde{b}(\cos \theta) F(\vet)\la \vet\ra^{-s}  \la v \ra^{-\frac{k}{p}} F(v) \left| F^{p-1}(v')- F^{p-1}(v) \right| \d \sigma\d\vet\d v.\end{split}
\end{equation*}
Using the inequality \eqref{eq:X2} gives that
\begin{equation}\label{eq:FX2}
F(v)\left|F^{p-1}(v')-F^{p-1}(v)\right| \leq \left|F^{\frac{p}{2}}(v')-F^{\frac{p}{2}}(v)\right|\,\left[F^{\frac{p}{2}}(v')+F^{\frac{p}{2}}(v)\right]\,,\end{equation}
where, we observe, if $p=2$ there is no need for the second term $F^{\frac{p}{2}}(v)=F(v)$, one simply keeps the estimate $F\left|F(v')-F(v)\right|$.
Moreover,
$$\tilde{b}(\cos \theta)=\sqrt{b(\cos\theta)}\sqrt{\sin^{2}\left(\tfrac{\theta}{2}\right)b(\cos\theta)}=\sqrt{b_{s}(\cos\theta)b_{s-1}(\cos\theta)}\,,$$
where $b_{s-1}(\cos \theta) \sim  b_{0}\theta^{1-2s}= b_{0}\theta^{-1-2(s-1)}$
satisfies \eqref{eq:b0s} with $s$ replaced by $s-1$ (in particular, $\tilde{b} \sim b_{s-\frac{1}{2}}$). Thus
\begin{multline*}
\mathscr{J}_{3,1}[f] \leq \int_{\R^{2d}\times\S^{d-1}}\sqrt{b_{s}(\cos\theta)}\left|F^{\frac{p}{2}}(v')-F^{\frac{p}{2}}(v)\right|\,|v-\vet|^{\frac{\g}{2}}\\
\times \sqrt{b_{s-1}(\cos\theta)}\left[F^{\frac{p}{2}}(v')+F^{\frac{p}{2}}(v)\right] |v-\vet|^{\frac{\g}{2}+s}\d\mu_f,
\end{multline*}
with $\d\mu_f=f(\vet)\la v\ra^{-s}\d\sigma\d\vet\d v.$
Using Young's inequality, it follows that, for $\delta >0$,
$$\mathscr{J}_{3,1}[f] \leq \delta\,\mathscr{D}[f,F^{\frac{p}{2}}] + \frac{1}{2 \delta} \Big(\mathcal{B}_{1}[f,F^{p}]+ \mathcal{B}_{2}[f,F^{p}]\ind_{p\neq 2}\Big),$$
where, for suitable functions $g$ and $h$, one introduces $$\mathcal{B}_{1}[g,h]=\int_{\R^{2d}\times \S^{d-1}}|v-\vet|^{\g+2s}b_{s-1}(\cos\theta)g(\vet) \la v\ra^{-2s}h (v')\d\sigma\d\vet\d v,$$
and
$$\mathcal{B}_{2}[g,h]=\int_{\R^{2d}\times \S^{d-1}}|v-\vet|^{\g+2s}b_{s-1}(\cos\theta)g(\vet) \la v\ra^{-2s}h (v)\d\sigma\d\vet\d v.$$
Recall that, for $p=2$, one does not use \eqref{eq:FX2} and the term $\mathcal{B}_{2}[f,F^{p}]$ is not required.  In the same way
\begin{multline*}
\mathscr{J}_{3,2}[f] \leq \int_{\R^{2d}\times\S^{d-1}}\sqrt{b_{s}(\cos\theta)}\left|F^{\frac{p}{2}}(v')-F^{\frac{p}{2}}(v)\right|\,|v-\vet|^{\frac{\g}{2}}\times\\\times \sqrt{b_{s-1}(\cos\theta)}\left[F^{\frac{p}{2}}(v')+F^{\frac{p}{2}}(v)\right]\langle \vet\rangle ^{\frac{k}{p}-s}\langle v\rangle^{s-\frac{k}{p}}\,|v-\vet|^{\frac{\g}{2}+s}\d {\mu_f},\end{multline*}
and Young's inequality shows now that 
$$\mathscr{J}_{3,2}[f] \leq \delta \,\mathscr{D}\left[f,F^{\frac{p}{2}} \right] + \frac{1}{2\delta} \left(
\mathcal{B}_{1}\left[\la \cdot \ra^{\frac{2k}{p}-2s}f,F^{p}\right]+\mathcal{B}_{2}\left[\la \cdot \ra^{\frac{2k}{p}-2s}f,F^{p}\right]\ind_{p\neq 2}\right).$$
Since $b_{s-1} \in L^1\left( \S^{d-1} \right)$, we use the regular change of variables $(v', v_*) \rightarrow (v, v_*)$, see for instance \cite[Proof of Lemma 1]{ADVW} or \cite[Proposition 2.3, equation (2.5)]{CG}, for estimating $\mathcal{B}_{1}$.  Thus, there exists $C_{b} >0$ depending only on $b$ such that, for nonnegative functions $g$ and $h$,
$$\mathcal{B}_{1}[g,h] + \mathcal{B}_{2}[g,h] \leq C_{b}\int_{\R^{2d}}|v-\vet|^{\g+2s}g(\vet)\,h(v)\d \vet\d v=C_{b}\int_{\R^{d}}\bm{c}_{\g+2s}[g](v)h(v)\d v.$$
Gathering these estimates, we obtain that for any $\delta >0$ there exists $C_{\delta} >0$ depending on $\delta,k,p$ such that
$$\mathscr{J}_{3}[f] \leq 2\delta\mathscr{D}\left[f,F^{\frac{p}{2}}\right] +C_{\delta}\sum_{\beta=0,\frac{k}{p}-s}\int_{\R^{3}}\bm{c}_{\g+2s}[\la \cdot \ra^{2\beta}f](v)F^{p}(v)\d v,$$
which gives the result after keeping the leading term and diminishing $\delta$ to $\frac{\delta}{2}$.\end{proof}
\subsection{Weighted fractional $\e$-Poincar\'e inequality} 
The estimates for $\mathscr{J}_{1}[f],$ $\mathscr{J}_{2}[f]$ and $\mathscr{J}_{3}[f]$ are of different nature.  Indeed, $\mathscr{J}_{1}[f]$ involve $\bm{c}_{\g}[f]$ (with the parameter $\g$ and without weight on $f$) whereas $\mathscr{J}_{i}[f]$ (for $i=2,3$) involve $\bm{c}_{\g+2}[\la \cdot \ra^{\ell}f]$, that is, a term with milder singularity $\bm{c}_{\g+2}$ but with weights acting on $f$.  In order to compare the various terms in an unified way, we need a version of the $\e$-Poincar\'e inequality \eqref{eq:estimatc} which handles weights and the milder parameter $\g+2$.  The milder singular term allows to compensate the action of the weights using the same parameter $q$ at the price of adding suitable $L^{1}$-moments. 

\begin{prop}[\textit{\textbf{Weighted fractional $\e$-Poincar\'e}}]
\label{prop:ePo-weight}
Assume that $d\in\N$, $d\ge 2$, $s \in (0,1]$  and $-d < \g+2s < 0$, and
\begin{equation}\label{eq:rangeq}
\frac{d}{d+2s+\g} < q  < \frac{d}{d+\g}. \end{equation}
For any $\beta >0$, $\bm{a} > \beta +|\g|$, and $\e >0$, there exist a positive constant $C_{\e} >0$ depending on $\beta,s,\nu,d,\bm{a}$ and $\e>0$  such that
\begin{multline}\label{eq:AA}
\int_{\R^{d}} \phi^{2}(v)\,\bm{c}_{\g+2s}\left[\langle \cdot \rangle^{\beta}g\right](v)\d v \leq \e\,\|\langle \cdot \rangle^{\frac{\g}{2}}\phi\|_{\dot{\mathbb{H}}^{s}}^{2}
+ C_{\e}\left(\|\langle \cdot \rangle^{\bm{a}}g\|_{L^1} +\|\langle \cdot \rangle^{|\g|}g\|_{L^q}^{\frac{s}{s-\nu}}\right)\| \phi\|_{L^2}^{2}\,,\end{multline}
holds for any $g \geq0$ and $\phi$ sufficiency smooth.
\end{prop}
\begin{rmq} Observe that the last term in the weighted $\e$-Poincar\'e inequality \eqref{eq:AA} is $\|\phi\|_{L^2}^{2}$ and differs from \eqref{eq:estimatc} which involves $\|\la \cdot\ra^{\frac{\g}{2}}\phi\|_{L^2}$. This comes from the estimate of $\mathcal{I}_{2}$ hereafter and, likely, can be improved.
\end{rmq}
\begin{rmq} Although the parameter involved in $\bm{c}_{\g+2s}$ is $\g+2s >\g$, the parameter $q$ in \eqref{eq:rangeq} is the one associated to the parameter $\g$.  We reconcile such disparity by purposely raising the singularity of the term $\bm{c}_{\g+2s}[\la \cdot\ra^{\beta}]$ to a term $\bm{c}_{\g+\alpha}[\la \cdot\ra^{\beta}g]$ for $\alpha\in(0,2 {s})$ sufficiently small.
\end{rmq}
\begin{proof}  Let the functions $\phi$ and $g\geq0$ be fixed.  We set $G_{\l}=\la \cdot \ra^{\l}g$ for any $\l \geq0$ and introduce 
$$
\mathcal{I}=\ds\int_{\R^{d}}\phi^{2}(v)\,\bm{c}_{\g+2s}\left[G_{\beta}\right]\d v=\mathcal{I}_{1}+\mathcal{I}_{2},$$
with 
\begin{equation*} 
\mathcal{I}_{1} =\ds\int_{|v-\vet| \leq1}|v-\vet|^{\g+2s}\phi^{2}(v)G_{\beta}(\vet)\d v\d\vet\,,
\quad \mathcal{I}_{2} =\int_{|v-\vet|>1}|v-\vet|^{\g+2s}\phi^{2}(v)G_{\beta}(\vet)\d v\d\vet. \end{equation*}  {Using that $\g+2s < 0$}, it holds
$$\mathcal{I}_{2} \leq \|G_{\beta}\|_{L^1}\, \|\phi\|^{2}_{L^{2}}.$$
Now, for some suitable $\alpha \in (0,2s)$ to be determined we observe that
$$\mathcal{I}_{1} \leq \int_{|v-\vet|\leq 1}\phi^{2}(v)|v-\vet|^{\g+\alpha}G_{\beta}(\vet)\d\vet\d v.$$
Since on the region $|v-\vet| \leq1$ it holds $\la \vet\ra\simeq \la v\ra$, there is $C_{\alpha} >0$ such that
$$G_{\beta}(\vet)=G_{\beta}(\vet)\langle \vet\ra^{-\alpha}\la \vet\ra^{\alpha} \leq C_{\alpha} \la v\ra^{-\alpha}\,G_{\beta+ \alpha}(\vet),$$
and, therefore,
$$\mathcal{I}_{1} \leq C_{\alpha}\int_{\R^{d}}\psi^{2}(v)\bm{c}_{\g+\alpha}[G_{\beta+\alpha}]\d v\,, \quad  \quad \psi(v)=\la v\ra^{-\frac{\alpha}{2}}\phi(v).$$ 
Then, use the $\e$-Poincar\'e inequality \eqref{eq:estimatc} corresponding to the parameter $\g+\alpha$ and some 
$$\frac{d}{d+2s+\g+\alpha} < q_{\alpha}  < \frac{d}{d+\g+\alpha},$$
from which we deduce that for any $\e >0$ there exists $C_{\e} >0$ depending only on $d, \g, \alpha,s,q_{\alpha}$ such that, for $\nu_{\alpha}=\frac{d-q_{\alpha}(d+\g+ {\alpha})}{2q_{\alpha}}$,
\begin{equation*}\begin{split} \label{eq:estimatc0}
\mathcal{I}_{1}  &\leq \e\,\left\|\langle \cdot \rangle^{\frac{\g+\alpha}{2}}\psi\right\|_{\dot{\mathbb{H}}^{s}}^{2}
+C_{\e}\left( \|G_{\beta+\alpha}\|_{L^1} +\left\|\langle \cdot \rangle^{|\g+\alpha|}G_{\beta+\alpha}\right\|_{L^{q_{\alpha}}}^{\frac{s}{s-\nu_{\alpha}}}\right)\int_{\R^{d}}\psi^{2}(v)\langle v\rangle^{\g+\alpha}\d v\\
&= \e\,\left\|\langle \cdot \rangle^{\frac{\g}{2}}\phi\right\|_{\dot{\mathbb{H}}^{s}}^{2}
+C_{\e}\left( \|G_{\beta+\alpha}\|_{L^{1}} +\left\|\langle \cdot \rangle^{|\g+\alpha|}G_{\beta+\alpha}\right\|_{L^{q_{\alpha}}}^{\frac{s}{s-\nu_{\alpha}}}\right)\|\langle \cdot \rangle^{\frac{\g}{2}}\phi\|_{L^2}^{2}.\end{split}
\end{equation*}
Let us choose $q$ as in \eqref{eq:rangeq} so that, similar to \eqref{eq:prHLS}, we can write
$$q=\frac{d}{d+\g+2\nu}, \qquad q_{\alpha}=\dfrac{d}{d+\g+\alpha+2\nu_{\alpha}}\,,$$ with $\nu$ and $\nu_{\alpha} \in [0,s]$.  Now, it is possible for a suitable choice of $\alpha$ and $q_{\alpha}$ (or equivalently $\nu_{\alpha}$) to estimate $\left\|\langle \cdot \rangle^{|\g+\alpha|}G_{\beta+\alpha}\right\|_{L^{q_{\alpha}}}^{\frac{s}{s-\nu_{\alpha}}}$ with $\left\|\langle \cdot \rangle^{|\g|}g\right\|_{L^q}^{\frac{s}{s-\nu}}$ and suitable moments of $g$. Indeed, using the interpolation \eqref{int-ineqp_bis} and recalling that $G_{\beta+\alpha}=\la \cdot \ra^{\beta+\alpha}g$ we deduce that
$$\left\|\langle \cdot \rangle^{|\g+\alpha|+\beta+\alpha}g\right\|_{L^{q_{\alpha}}} \leq \left\|\langle \cdot \rangle^{\bm{a}}g\right\|_{L^1}^{1-\theta}\,\left\|\langle \cdot \rangle^{|\g|}g\right\|_{L^q}^{\theta}\,,$$
with
\begin{equation}\label{eq:1qalp}
\frac{1}{q_{\alpha}}=1-\theta+\frac{\theta}{q}, \qquad |\g+\alpha|+\beta+\alpha=\bm{a}(1-\theta)+|\g|\theta.
\end{equation}
In other words,
$$\left\|\langle \cdot \rangle^{|\g+\alpha|}G_{\beta+\alpha}\right\|_{L^{q_{\alpha}}}^{\frac{s}{s-\nu_{\alpha}}} \leq \left\|\langle \cdot \rangle^{\bm{a}}g\right\|_{L^1}^{\frac{s(1-\theta)}{s-\nu_{\alpha}}}\,\left\|\langle \cdot \rangle^{|\g|}g\right\|_{L^q}^{\frac{s\theta}{s-\nu_{\alpha}}}.$$
Moreover, we can pick $\alpha \in (0,2)$ sufficiently small so that $\frac{s\theta}{s-\nu_{\alpha}} <\frac{s }{s-\nu}$. From \eqref{eq:1qalp} one deduces
$$\theta=\frac{\g+\alpha+2\nu_{\alpha}}{\g+2\nu}=1+\frac{\alpha+2(\nu_{\alpha}-\nu)}{\g+2\nu},$$
where we notice that $\g+2\nu < 0$. Therefore, 
$$\theta \in (0,1) \iff 2\nu < \alpha+2\nu_{\alpha} < |\g|.$$
One can choose, for example, $\nu_{\alpha}=\nu$ and get $0<\theta=1+\frac{\alpha}{\g+2\nu} < 1$ with $\alpha < |\g+2\nu|$. Such a choice of $\theta$ ensures that
$$\left\|\langle \cdot \rangle^{|\g+\alpha|}G_{\beta+\alpha}\right\|_{L^{q_{\alpha}}}^{\frac{s}{s-\nu_{\alpha}}} \leq \left\|\langle \cdot \rangle^{\bm{a}}g\right\|_{L^1}^{\frac{s(1-\theta)}{s-\nu}}\,\left\|\langle \cdot \rangle^{|\g|}g\right\|_{L^q}^{\frac{s\theta}{s-\nu}} \leq (1-\theta)\left\|\langle \cdot \rangle^{\bm{a}}g\right\|_{L^1}^{\frac{s}{s-\nu}}+\theta\left\|\langle \cdot \rangle^{|\g|}g\right\|_{L^q}^{\frac{s}{s-\nu}}$$
thanks to Young's inequality. Notice that \eqref{eq:1qalp} yields, with such choice of $\theta$, 
$$\bm{a}=\frac{|\g+2\nu|}{\alpha}\beta+|\g|.$$ The optimal choice of $\alpha$ corresponds to $\frac{|\g+2\nu|}{\alpha}$ arbitrarily close to $1$ which justifies the choice of arbitrary $\bm{a} > \beta+|\g|$.  This choice is independent of both $\g$ and $\nu$, and so is $q$.
\end{proof} 
 
 \subsection{Appearance and propagation of weighted $L^{p}$-estimates} Gathering the results of the two previous sections, we deduce the following theorem.
\begin{theo}\label{mpsweight}
 Let $f_{\rm in} \in \mathcal{Y}_{\rm in}$ be given and let $f=f(t,v)$ be a classical  solution to Eq. \eqref{eq:Boltz} with 
 $$B(u,\sigma)=|u|^{\g}b(\cos \theta)\,,$$
where $b(\cdot)$ satisfies \eqref{eq:b0s} and $\g+2s <0.$  Given $p \in (1,\infty)$ and $k \geq0$ assume that
\begin{equation}\label{eq:ProSer}\langle \cdot \rangle^{|\g|}\,f \in L^{r}\left([0,T]\,;\,L^{q}(\R^{d})\right) \:\: \text{ with }\:\: \frac{2s}{r}+\frac{d}{q}=2s+d+\g\,,\end{equation}
for some $T >0$, where $r \in (1,\infty)$, $q \in \left(\frac{d}{2s+d+\g},\infty \right)$. 
Then, the following statements hold.
\begin{enumerate}
\item [(a)] \textnormal{\textbf{(Propagation of weighted Lebesgue norm)}} If 
$$f_{\rm in} \in L^{p}_{k}(\R^{d}) \cap L^{1}_{\ell}(\R^{d}) \qquad \text{ for }\qquad \ell > 2\frac{k}{p}+|\g|$$
then
\begin{equation}\label{eq:propaLp}
\left\|f(t)\right\|_{L^{p}_{k}}^{p} + \int_{0}^{t} \left\|\langle \cdot \rangle^{\frac{k+\g}{2}}f(\tau)^{\frac{p}{2}}\right\|_{\dot{\mathbb{H}}^{s}}^{2}\d\tau\leq \bm{C}\left(T,\lm_{{\ell}}(f_{\rm in})\right) \|f_{\rm in}\|_{L^{p}_{k}}^{p}\end{equation}
for a explicit constant $\bm{C}\left(T,\lm_{\bm{a}}(f_{\rm in})\right)$ depending on $\ds\int_{0}^{T}\left\|\langle \cdot \rangle^{|\g| }f(\tau)\right\|_{L^q}^{r}\d\tau$ as well as $T$, $k$, $p$, $q$, $s$, $\g$, $d$, $\lm_{\ell}(f_{\rm in})$, but not on $\|f_{\rm in}\|_{L^{p}_{k}}$.
\item [(b)] \textnormal{\textbf{(Appearance of weighted Lebesgue norm)}} For $f_{\rm in}$ with sufficient statistical moments,
$$f_{\rm in} \in L^{1}_{\eta_{p,k}}(\R^{d}), \qquad \eta_{p,k}:=\frac{|\g|d}{2s}\left(1-\frac{1}{p}\right) +\frac{k}{p}\,,$$
then, the solution $f=f(t,v)$ satisfies the estimate, for $t \in (0,T]$,
\begin{equation}\label{Lp-Kpkq}
\|f(t)\|_{L^{p}_{k}} \le \bm{K}_{p,B,q,T} \, t^{- \frac{d}{2s} \left(1 - \frac1{p} \right)}\sup_{\tau \in [0,T]}\lm_{\eta_{p,k}}(\tau),\end{equation}
for a explicit constant $\bm{K}_{p,B,q,T}$ depending on $\ds\int_{0}^{T}\|\langle \cdot\rangle^{|\g|}f(t)\|_{L^q}^{r}  \d t$ as well as $p,q,s,\g,d,$ $\varrho_{\rm in},E_{\rm in},H(f_{\rm in})$.
\end{enumerate}
\end{theo}
 
\begin{proof} Recalling that $\lM_{k,p}(t)=\|f(t)\|_{L^{p}_{k}}^{p}$ and
$$\frac{1}{p}\dfrac{\d}{\d t}\lM_{k,p}(t)=\mathscr{J}_{1}[f(t)]+\mathscr{J}_{2}[f(t)]+\mathscr{J}_{3}[f(t)],$$
we deduce from \eqref{eq:J1}--\eqref{eq:J2}--\eqref{eq:J3} that, for any $\delta >0$,
\begin{multline*}
\frac{1}{p}\dfrac{\d}{\d t}\lM_{k,p}(t) +\frac{1}{\max(p,p')}\mathscr{D}\left[f(t),F^{\frac{p}{2}}(t) \right] \leq \frac{1}{p'}\int_{\R^{d}}\bm{c}_{\g}[f(t)](v)F^{p}(t,v)\d v \\
+ c_{k,p}(B)  \int_{\R^{d}}\bm{c}_{\g+2s} \left[\langle \cdot\rangle^{\frac{k}{p}-2s}f(t) \right]  F^{p}(t,v)\d v 
+\delta\mathscr{D} \left[f(t),F^{\frac{p}{2}}(t) \right] \\
+C_{\delta}(B)\int_{\R^{d}} \bm{c}_{\g+2s}\left[\langle \cdot \rangle^{2\frac{k}{p}-2s}f(t) \right](v)F^{p}(t,v) \d v,\end{multline*}
where $F(t,v)=\langle v\rangle^{\frac{k}{p}}f(t,v)$. Choose $\delta=\delta_{0}=\frac{1}{2\max(p,p')}$.  Since $\frac{2k}{p}-2s > \frac{k}{p}-2s$, notice that there exists $\bm{C}_{k,p} >0$ depending on $p,k,B$ such that
\begin{multline*}
\dfrac{\d}{\d t}\lM_{k,p}(t) +\frac{p}{2\max(p,p')}\mathscr{D} \left[f(t),F^{\frac{p}{2}}(t) \right]\leq (p-1)\int_{\R^{d}}\bm{c}_{\g}[f(t)](v)F^{p}(t,v)\d v\\
+\bm{C}_{k,p}\int_{\R^{d}} F^{p}(t,v)\bm{c}_{\g+2s}\left[\langle \cdot \rangle^{2\frac{k}{p}-2s}f(t) \right](v) \d v.\end{multline*}
Recalling that
$$\lD_{s,\g+k,p}(t)=\left\|\langle \cdot \rangle^{\frac{\g}{2}}F^{\frac{p}{2}}(t)\right\|_{\dot{\mathbb{H}}^{s}}^{2} \qquad \text{and}\qquad F=\langle \cdot \rangle^{\frac{k}{p}}f,$$
we deduce from  Lemma \ref{lem:coerc} that there exists $c_{p},C_{p}>0$ such that
$$\frac{p}{2\max(p,p')}\mathscr{D} \left[f(t),F^{\frac{p}{2}}(t) \right] \geq c_p\lD_{s,\g+k,p}(t)-C_p \lM_{k+\g,p}(t).$$
Therefore,
\begin{multline}\label{eq:lMkpTT}
\dfrac{\d}{\d t}\lM_{k,p}(t) +c_{p}\lD_{s,k+\g,p}(t)\leq C_{p}\lM_{k+\g,p}(t)+(p-1)\int_{\R^{d}}\bm{c}_{\g}[f(t)](v)F^{p}(t,v)\d v\\
+\bm{C}_{k,p}\int_{\R^{d}} F^{p}(t,v)\bm{c}_{\g+2{s}}\left[\langle \cdot \rangle^{2\frac{k}{p}-2}f(t) \right](v) \d v.
\end{multline}
The integral involving $\bm{c}_{\g}$ is estimated using \eqref{eq:estimatc} with the choices $\phi=F^{\frac{p}{2}}$ and $g= f $, so that, for all $\e >0$ there is a constant $C_{\e} >0$ such that
$$\int_{\R^{d}}\bm{c}_{\g}[f(t)](v)F^{p}(t,v)\d v\leq \e\lD_{s,\g+k,p}(t)+ C_{\e}\left(\|f(t)\|_{L^1}+\left\|\langle \cdot \rangle^{|\g|}f(t)\right\|_{L^q}^{r}\right)\lM_{k+\g,p}(t),$$ 
where we recall $r=\frac{s}{s-\nu}$, as in the proof of Theorem \ref{mps_g}.  Similarly, we can use Proposition \ref{prop:ePo-weight} with $\beta=2\frac{k}{p}$ and $\bm{a}=\ell >2\frac{k}{p}+|\g|$, with the same choice of $q$, $r$, to deduce that for all $\e >0$, there is a constant $C_{\e} >0$ such that
\begin{multline*}
\int_{\R^{d}}F^{p}(t,v)\bm{c}_{\g+2\textcolor{blue}{s}} \left[\langle \cdot\rangle^{2\frac{k}{p}}f(t) \right](v) \d v
\leq \e\lD_{s,\g+k,p}(t)+C_{\e}\left(\left\|\langle \cdot \rangle^{|\g|}f(t)\right\|_{L^q}^{r}+\lm_{\ell}(t)\right)\lM_{k,p}(t).\end{multline*}
Choosing $\e >0$ sufficiency small, we deduce that
$$\dfrac{\d}{\d t}\lM_{k,p}(t) +\frac{c_{p}}{2}\lD_{s,k+\g,p}(t) \leq \bm{\Lambda}(t)\lM_{k,p}(t),$$
with
$$\bm{\Lambda}(t)=\bm{C}_{1}\left(1+\left\|\langle \cdot \rangle^{|\g|}f(t)\right\|_{L^q}^{\frac{s}{s-\nu}}+\lm_{\ell}(t)\right).$$
Recall that under assumption $f_{\rm in} \in L^{1}_{\ell}(\R^{d})$ one has that $\sup_{t\in[0,T]}\lm_{\ell}(t) \leq C_{T}$ and $\bm{\Lambda} \in L^{1}(0,T)$ under the Prodi-Serrin condition \eqref{eq:ProSer}. Setting
$$\bm{C}\left(T,\lm_{\ell}(f_{\rm in})\right)=\exp\left\{\int_{0}^{T}\bm{\Lambda}(\tau)\d\tau\right\},$$
we deduces the propagation result. In order to prove the appearance of $L^{p}_{k}$-norm, copycat the proof of Theorem \ref{mps_g} where all estimates are computed for $F(t,v)=\la \cdot \ra^{\frac{k}{p}}f(t,v)$ instead of $f(t,v)$. The estimates in the proof of Theorem \ref{mps_g} show then
$$\lD_{s,k+\g,p}(t) \geq C_{\mathrm{Sob},s}^{-2}\|F(t)\|_{L^{1}_{\eta_{p}}}^{-\frac{2sp}{d(p-1)}}\lM_{k,p}^{1+\frac{2s}{d(p-1)}}=C_{\mathrm{Sob},s}^{-2}\lm_{\eta_{p}+\frac{k}{p}}(t)^{-\frac{2sp}{d(p-1)}}\lM_{p}(t)^{1+\frac{2s}{d(p-1)}}.$$
This proves the result as in Theorem \ref{mps_g}, since $\eta_{k,p}=\eta_{p}+\frac{k}{p}$.  
\end{proof}

\section{Stability and uniqueness of solutions}\label{sec:unique}
In this section we work in the physical dimension $d=3$.
\subsection{Useful estimates} To prove the stability of solutions to the Boltzmann equation, we first derive suitable estimates for $\Q$ and establish an alternative to Proposition \ref{prop:ePo} in order to estimate $\bm{c}_\gamma[f]$. We begin with the following lemma.
 \begin{lem}
 Consider $\alpha \in(-3, 1]$, $s \in (0, 1)$, and $-\frac{3}{2} < \alpha + 2s < 0$.  For any $\ell > \frac{3}{2} + (\alpha + 2 s)$ and $\beta \in \R$, we have that
 	\begin{equation}
 		\label{eq:conv_soft}
 		\int_{ \R^3 } \bm{c}_\alpha [\chi] \psi^2 \d v \leq  C_{\beta,\ell}\| \la \cdot \ra^{|\alpha| + \beta + \ell} \chi \|_{L^2} \left(  \|\langle \cdot\rangle^{\frac{\alpha - \beta}{2}} \psi \|^2_{ \mathbb{H}^s} + \|\langle \cdot\rangle^{\frac{\alpha}{2}} \psi \|^2_{ L^2 }  \right)\,,
 	\end{equation}
for an explicit constant $C_{\beta,\ell} >0$ depending on $d,\alpha,s,\beta$ and $\ell$. Furthermore, if $\alpha + 2 s > 0$ we obtain that
 	\begin{equation}
 		\label{eq:conv_mod_soft}
 		\int_{ \R^3 } \bm{c}_\alpha [\chi] \psi^2 \d v \leq C_{\beta,\ell}\| \la \cdot \ra^{|\alpha| + \beta} \chi \|_{L^1} \left(  \|\langle \cdot\rangle^{\frac{\alpha - \beta}{2}} \psi \|^2_{ \mathbb{H}^s} + \|\langle \cdot\rangle^{\frac{\alpha}{2}} \psi \|^2_{ L^2 }  \right) \, .
 	\end{equation}
\end{lem}
\begin{proof} Writing
$$\int_{\R^{3}}\bm{c}_{\alpha}[\chi]\psi^{2}\d v=\int_{\R^{3}}\chi(v)\d v\int_{\R^{3}}|v-\vet|^{\alpha}\psi^{2}(\vet)\d\vet\,,$$
we can divide the internal integral in small and large relative velocities as
$$\int_{ \R^3 } |v-\vet |^{\alpha} \psi^2(\vet) \d \vet =\int_{ |v-\vet| \le 1 } |v-\vet |^{\alpha} \psi^{2}(\vet) \d\vet + \int_{ |v-\vet| \ge 1 } |v-\vet |^{\alpha} \psi^{2}(\vet) \d \vet\,.$$
Using then Peetre's inequality in the form $1 \lesssim \langle v\rangle^{|\alpha|}\langle \vet\rangle^{\alpha}\langle v-\vet\rangle^{\alpha}$, we deduce that
\begin{equation*}\begin{split}
\int_{ \R^3 } |v-\vet |^{\alpha} \psi^2(\vet) \d \vet &\lesssim \la v \ra^{|\alpha|} \left( \int_{ |v-\vet| \ge 1 } \la \vet \ra^{\alpha} \psi^2(\vet) \d \vet
 		+ \int_{ |v-\vet| \le 1 } |v-\vet |^{\alpha} \la \vet \ra^{\alpha} \psi^{2}(\vet) \d \vet\right) \\
 		\lesssim & \la v \ra^{|\alpha|} \left( \|\langle \cdot \rangle^{\frac{\alpha}{2}} \psi \|^2_{2}
 		+ \int_{ |v-\vet| \le 1 } |v-\vet |^{\alpha} \la \vet \ra^{\alpha} \psi^{2}(\vet) \d\vet \right) \, .
 	\end{split}\end{equation*}
Therefore, observing that $\la v \ra \approx \la \vet \rangle$ whenever $|v-v_*| \le 1$
$$\int_{\R^{3}}\bm{c}_{\g}[\chi]\psi^{2}\d v \lesssim \|\langle \cdot \rangle^{|\alpha|}\chi\|_{L^1} \|\langle \cdot \rangle^{\frac{\alpha}{2}} \psi \|^2_{2} + \int_{ |v-\vet| \le 1 } |v-\vet |^{\alpha} \la \vet \ra^{\alpha - \beta} \psi^{2}(\vet) \la v \ra^{|\alpha| + \beta} \chi(v) \d \vet\d v \, .$$ For $-\frac{3}{2}< \alpha+2s<0$, we  use the Hardy-Littlewood-Sobolev inequality with the choice $\frac{1}{p} = 1 + \frac{\alpha+2s}{3}$ and $\frac{1}{r} = 1-\frac{2s}{3}$ to deduce that
$$\int_{\R^{3}}\bm{c}_{\g}[\chi]\psi^{2}\d v \lesssim \|\langle \cdot \rangle^{|\alpha|}\chi\|_{L^1} \|\langle \cdot \rangle^{\frac{\alpha}{2}} \psi \|^2_{L^{2}} + \| \la \cdot \ra^{|\alpha| + \beta} \chi \|_{L^p} \| \la \cdot \ra^{\frac{\alpha - \beta}{2}} \psi \|_{L^{2r} }^2\,.$$
Observing that $p <2$ and choosing $\l > \frac{3}{2}+\alpha+2s$ one has that $L^2_{2 \ell}(\R^{3}) \hookrightarrow L^p(\R^{3})$.   We also recall the embedding $\mathbb{H}^s \hookrightarrow L^{ 2 r }(\R^{3})$, see \cite[Theorem 1.66]{chemin}, to obtain  \eqref{eq:conv_soft}.  In the case of moderately soft potentials, that is $\alpha + 2s > 0$, the result is obvious if $\alpha \ge 0$, so we focus in the case $\alpha < 0$.  Using H\"older's inequality
\begin{multline*}
\int_{ |v-\vet| \le 1 } |v-\vet |^{\alpha} \la \vet \ra^{\alpha - \beta} \psi^{2}(\vet) \la v \ra^{|\alpha| + \beta} \chi(v) \d \vet\d v\\
\leq \| \la \cdot \ra^{|\alpha| + \beta} \chi \|_{L^1} \left\| | \cdot |^{-\alpha}\ind_{|\cdot| \leq 1} \right\|_{L^{\frac{3}{2s}}} \| \la \cdot \ra^{\frac{\alpha - \beta}{2}} \psi \|_{L^{2r} }^2\, ,
\end{multline*}
with $\frac{1}{r}=1-\frac{2s}{3}$ and thus, once again, we deduce \eqref{eq:conv_mod_soft}. This concludes the proof.
\end{proof}
Let us revisit and extend estimates in \cite[Proposition 3.1]{AMSY} to the soft potential case. The proof of the result is postponed to Appendix \ref{app:AMSY}. 
\begin{prop}\label{prop:AMSY} Given $0 < s < 1$, $\g+2s <0$ and $\l >\max\left(6,\frac{9+\g}{2}+2s\right)$, there exists a constant $\kappa_{\l,b} >0$ depending only on $\l$ and $b(\cdot)$ such that
\begin{multline*}
\int_{\R^{6}\times \S^{2}}b(\cos\theta)|v-\vet|^{\g}\left(\langle v'\rangle^{\l}-\langle v\rangle^{\l}\right)f_{\ast} g\,h'\d v\d\vet\d\sigma\\
\leq 
 \kappa_{\l,b}\,\left(\int_{\R^{3}}\bm{c}_{\g}\left[\langle \cdot \rangle\,|g|\right](v)\langle v\rangle^{2\l}f^{2}(v)\d v\right)^{\frac{1}{2}}
\left(\int_{\R^{3}}\bm{c}_{\g}\left[\langle \cdot \rangle\,|g|\right](v)h^{2}(v)\d v\right)^{\frac{1}{2}}\\
+\kappa_{\l,b}\left(\int_{\R^{3}}\bm{c}_{\g}\left[\langle \cdot \rangle^{4}\,|f|\right](v)\langle v\rangle^{2\l}g^{2}(v)\d v\right)^{\frac{1}{2}}
\left(\int_{\R^{3}}\bm{c}_{\g}\left[\langle \cdot \rangle^{4}\,|f|\right](v)h^{2}(v)\d v\right)^{\frac{1}{2}}\\
+ \kappa_{\l,b}\sqrt{\mathscr{D}_{\g+2}\left[\langle \cdot \rangle|f|\,,\,\langle \cdot \rangle^{\l-2}g\right]}\left(\int_{\R^{3}}\bm{c}_{\g}[\langle \cdot \rangle|f|]h^{2}\d v\right)^{\frac{1}{2}}\,,\end{multline*}
for all smooth $f,g,h$ where $\mathscr{D}_{\g+2}$ is defined in Remark \ref{rmq:Dgamma}.
\end{prop}
A consequence of this inequality is the following commutator estimate.  Define,
\begin{equation}
\label{eq:commutator_implicit_def}
\mathcal{R}_{k}(f,g,\psi)=\la \Q(f, \la \cdot \ra^{ \frac{k}{2} } g), \la \cdot \ra^{ \frac{k}{2} } \psi  \ra_{L^2} - \la \Q(f, g) , \psi \ra_{L^{2}_k} \, ,\qquad k\geq0\,.
\end{equation}
\begin{prop}[\textbf{\textit{Commutator estimate}}]\label{cor:Comm}
Let $s \in (0, 1)$ and $- \frac{3}{2} < \gamma + 2 s < 0$ and $k > 11 + 4s$.  There exists a constant $C_{k} >0$, depending on $k,\g,s$, such that 
\begin{multline}
 	 	\label{eq:comm_general}
\Big|\mathcal{R}_{k}(f,g,\psi)\Big| \leq C_{k} \|f \|_{L^1_{\g+3+2s}}^{\frac{1}{2}} \| \langle \cdot\rangle^{\frac{\gamma+k}{2}}\,g \|_{ \mathbb{H}^s} \left(\int \bm{c}_\gamma [ \la \cdot \ra |f| ] \langle v\rangle^{k}\psi^2 \d v \right)^{\frac{1}{2}}
\\
+C_{k}	\underset{a\neq b}{\sum_{a,b \in \{|f|,|g|\}}}\,\left(\int_{\R^{3}}\bm{c}_{\g}\left[\langle \cdot \rangle^{4}\,a\right](v)\langle v\rangle^{k}b^{2}(v)\d v\right)^{\frac{1}{2}}
\left(\int_{\R^{3}}\bm{c}_{\g}\left[\langle \cdot \rangle^{4}\,a\right](v)\langle v\rangle^{k}\psi^{2}(v)\d v\right)^{\frac{1}{2}}\,.
\end{multline}
Furthermore, for any $\varepsilon >0$
\begin{multline*}
\Big| \mathcal{R}_k(f,\psi,\psi) \Big| + \Big| \mathcal{R}_{k-\g}(\psi, g, \la \cdot \ra^{\frac{\g}{2}} \psi) \Big| \\
\le \varepsilon\left(1+\|f\|_{L^{1}_{\g+3+2s}}^{\frac{1}{2}} + \| g \|_{L^2_k} \right)\left\|\langle \cdot\rangle^{\frac{\g+k}{2}}\psi\right\|_{\mathbb{H}^{s}}^{2}+\Big(\bm{\Lambda}_{\e}[f] + \bm{\Lambda}_{\e}[ g]\Big)\,\|\psi\|_{L^2_k}^{2}\,,
\end{multline*}
for any smooth $f$, $g$ and $\psi$.   Here, for any smooth function $\varphi$, 
\begin{equation}\label{eq:Lamb}
\bm{\Lambda}_{\e}[\varphi]=\bm{C}_{\e} \left(1+\|\varphi\|_{L^{1}_{\g+3+2s}}^{\frac{1}{2}}\right)
\left(\|\varphi\|_{L^{1}_{4+|\g|}}  + \left\|\la \cdot \ra^{\frac{k}{2}} \varphi \right\|_{{\mathbb{H}}^{s} }^{2} + \|\varphi\|_{L^{2}_{2( 4+|\g| ) }}^{\frac{s}{s-\nu}}\right)\,,\end{equation}
for some positive $\bm{C}_{\e} >0$ depending on $k,s,\g$ and $\nu=-\frac{3+2\g}{4}.$
\end{prop}
\begin{proof} Observing that
$$\mathcal{R}_{k}(f,g,\psi)=-\int_{\R^{6}\times \S^{2}}b(\cos\theta)|v-\vet|^{\g}\left(\langle v'\rangle^{\frac{k}{2}}-\langle v\rangle^{\frac{k}{2}}\right)f_{\ast} g\,\langle v'\rangle^{\frac{k}{2}}\psi'\d v\d\vet\d\sigma\,,
$$
we can apply Proposition \ref{prop:AMSY} with $\l=\frac{k}{2}$ and $h=\langle \cdot \rangle^{\frac{k}{2}}\psi$ to deduce that there exists a constant $C_{k}=\kappa_{\l,b}$ such that, 
\begin{multline}\label{eq:Rkfgpsi}
\left|\mathcal{R}_{k}(f,g,\psi)\right| \leq  C_{k}\sqrt{\mathscr{D}_{\g+2}\left[\langle \cdot \rangle|f|\,,\,G\right]}\left(\int_{\R^{3}}\bm{c}_{\g}[\langle \cdot \rangle|f|]\langle v\rangle^{k}\psi^{2}(v)\d v\right)^{\frac{1}{2}}\\
 + C_{k}\underset{a\neq b}{\sum_{a,b \in \{|f|,|g|\}}}\,\left(\int_{\R^{3}}\bm{c}_{\g}\left[\langle \cdot \rangle^{4}\,a\right](v)\langle v\rangle^{k}b^{2}(v)\d v\right)^{\frac{1}{2}}
\left(\int_{\R^{3}}\bm{c}_{\g}\left[\langle \cdot \rangle^{4}\,a\right](v)\langle v\rangle^{k}\psi^{2}(v)\d v\right)^{\frac{1}{2}}\,,\end{multline}
with
$$G(v)=\langle v\rangle^{\frac{k}{2}-2}g(v)=\langle v\rangle^{-2} \left( \la v \ra^{\frac{k}{2}} g(v) \right).$$
  We estimate the term $\mathscr{D}_{\g+2}\big[\langle \cdot \rangle|f|,G\big]$ by comparing it to $\Q_{\g+2}$ (recall the notations introduced in Remark \ref{rmq:Dgamma}).  Using  the the weak form \eqref{eq:weak} together with the notations of Lemma \ref{lem:coerc} 
$$\left\la \Q_{\gamma+2} \left( \la v \ra |f|, G \right) , G\right\ra= - \frac{1}{2} \mathscr{D}_{\gamma+2} \left[ \la \cdot \ra |f| , G \right] + C_b \int_{\R^3} \bm{c}_{\gamma+2} \left[  \la \cdot \ra |f| \right] G^{2}(v) \d v \,$$
where we wrote $G(v')-G(v)=\frac{1}{2}\left[(G(v')^{2}-G(v)^{2}\right]-\frac{1}{2}\left[G(v')-G(v)\right]^{2}$ and used also Lemma \ref{lem:canc} and Remark \ref{nb:canc}. \color{black}
On the one hand, noticing that $\g+2+2s >0$, we deduce from \cite[Theorem 1.1, Eq. (1.14)]{He}, as well as \cite[Remark 1.5]{He} since $\g+2  >-\frac{3}{2}$, that
$$\left| \left\la \Q_{\gamma+2} ( \la \cdot \ra |f| , G),G \right\ra \right| \lesssim \|\langle \cdot\rangle f \|_{L^1_{\alpha}} \| \la \cdot \ra^{w_1} G\|_{\mathbb{H}^{a} } \| \la \cdot \ra^{w_2} G\|_{\mathbb{H}^{b} }\,,$$
where $a+b=2s$ and $w_{1}+w_{2}=\g+2+2s$, $\alpha=\g+2+2s+(-w_{1})^{+}+(-w_{2})^{+}.$ Choosing $a=b=s$ and $w_{1}=w_{2}=\frac{\g}{2}+1+s$, we deduce that
$$\left| \left\la \Q_{\gamma+2} ( \la \cdot \ra |f| , G),G \right\ra \right| \lesssim \| f \|_{L^1_{\g+3+2s}} \| \la \cdot \ra^{ \frac{\g}{2}+1+s } G\|_{\mathbb{H}^{s}}^{2}\lesssim \| f \|_{L^1_{\gamma +3 + 2s } } \| \la \cdot \ra^{\frac{\g+k}{2}} g \|^2_{ \mathbb{H}^s } \,,$$
where, we observe, since $s \in (0, 1)$
$$\| \la \cdot \ra^{ \frac{\g}{2}+1+s } G\|_{\mathbb{H}^{s}} = \| \la \cdot \ra^{1 + s - 2} \la \cdot \ra^{\frac{\g+k}{2}} g \|^2_{ \mathbb{H}^s }  \lesssim \| \la \cdot \ra^{\frac{\g+k}{2}} g \|^2_{ \mathbb{H}^s} \, .$$ On the other hand, using \eqref{eq:conv_mod_soft} and the fact that $(\gamma + 2) + 2 s > 1$, we have similarly
 	$$\int_{\R^3} \bm{c}_{\gamma+2} \left[ \langle \cdot \rangle\,|f| \right] G^{2}(v) \d v \lesssim \|f \|_{L^1_{1+|2+\gamma|}}\| \la \cdot \ra^{ \frac{\g}{2}+1 } G\|_{\mathbb{H}^{s}}^{2} \lesssim \|f\|_{L^{1}_{1+|2+\gamma|}}\,\| \la \cdot \ra^{ \frac{\gamma+k}{2} } g \|^2_{\mathbb{H}^s} \, .$$
 These estimates yield
 	$$\mathscr{D}_{\gamma+2} \left[ \la\cdot \ra |f|, G \right] \lesssim \| f \|_{L^1_{\gamma +3 + 2s } } \| \la \cdot \ra^{ \frac{\g+k}{2} } g \|^2_{ \mathbb{H}^s } \,.$$We deduce from this and \eqref{eq:Rkfgpsi} the general commutator estimate \eqref{eq:comm_general}.
	
\medskip
We now turn to the case $f = \psi$ and then $g = \psi$. For $g=\psi$, the sum in \eqref{eq:comm_general} is equal to $S=S_{1}+S_{2}$ with
$$S_{1}=C_{k} \int_{\R^{3}}\bm{c}_{\g}\left[\langle \cdot \rangle^{4}\,|f|\right](v)\langle v\rangle^{k}\psi^{2}(v)\d v\,,$$
and
$$S_{2}=C_{k}\left(\int_{\R^{3}}\bm{c}_{\g}\left[\langle \cdot \rangle^{4}\,|\psi|\right](v)\langle v\rangle^{k}f^{2}(v)\d v\right)^{\frac{1}{2}}
\left(\int_{\R^{3}}\bm{c}_{\g}\left[\langle \cdot \rangle^{4}\,|\psi|\right](v)\langle v\rangle^{k}\psi^{2}(v)\d v\right)^{\frac{1}{2}}.$$
Using \eqref{eq:conv_soft} with $\beta = 0$ and $k \geq 2 ( 4+\l+|\g| )$, that is $k > 11 + 4s$, and Young's inequality, there exists a constant $C'_{k} >0$ such that 
\begin{equation*}\begin{split}
S_{2} &\leq C'_{k} \| \la \cdot \ra^{ 4+\l+|\g| } \psi\|_{L^{2} }\|\langle \cdot\rangle^{\frac{\g+k}{2}}f\|_{\mathbb{H}^{s}}\|\langle\cdot\rangle^{\frac{\g+k}{2}}\psi\|_{\mathbb{H}^{s}}\\
& \leq  \varepsilon\|\langle\cdot\rangle^{\frac{\g+k}{2}}\psi\|_{\mathbb{H}^{s}}^{2}\,+\frac{( C'_{k} )^2 }{4\varepsilon}\|\psi\|_{L^{2}_{k}}^{2}\,\|\langle\cdot\rangle^{\frac{\g+k}{2}}f\|_{\mathbb{H}^{s}}^{2}\end{split}\end{equation*}
for all $\varepsilon >0$.  Now, according to Proposition \ref{prop:ePo} and since $\g +2s >-\frac{3}{2}$ ensures that $2 > \frac{3}{3+\g+2s}$, we can deduce \footnote{Notice that if $2 < \frac{3}{3+\g},$ i.e. $\g < -\frac{3}{2}$, one can apply directly Prop. \ref{prop:ePo}.  On the contrary, if $2 \geq \frac{3}{3+\g}$, then one applies Prop. \ref{prop:ePo} with some $\frac{3}{2+\g+2s} < q < \frac{3}{3+\g} \leq 2$ and use a simple interpolation
$$\|\la \cdot \ra^{|\g|+4}|f|\|_{L^{q}} \leq (1-\theta)\|\la \cdot \ra^{|\g|+4}|f|\|_{L^{1}}+\theta\|\la \cdot \ra^{|\g|+4}|f|\|_{L^{2}}, \qquad \theta=2-\frac{2}{q}\,,$$
to conclude.}
$$
S_{1} \leq \varepsilon \left\|\langle \cdot \rangle^{\frac{\g+k}{2}}\psi\right\|_{{\mathbb{H}}^{s}}^{2}
+C_{0}\left( \|f\|_{L^{1}_{4}}  +\e^{-\frac{\nu}{s-\nu}}\left\|\langle \cdot \rangle^{|\g|+4}f\right\|_{L^2}^{\frac{s}{s-\nu}}\right)\|\langle \cdot \rangle^{\frac{k}{2}}\psi\|_{L^2}^{2},
$$
which gives, for all $\varepsilon >0$,
$$S \leq 2\varepsilon\, \left\|\langle \cdot \rangle^{\frac{\g+k}{2}}\psi\right\|_{{\mathbb{H}}^{s} }^{2} + C_{\e}\,\left(\|f\|_{L^{1}_{4+|\g|}}  + \left\|\langle \cdot \rangle^{\frac{\g+k}{2}}f\right\|_{{\mathbb{H}}^{s}}^{2} + \| f\|_{L^{2}_{ 2 (4+|\g|) } }^{\frac{s}{s-\nu}}\right)\,\| \psi\|_{L^2_k}^{2}$$
for some $C_{\e} >0$ and $\nu=-\frac{3+2\g}{4}.$ The second term in \eqref{eq:comm_general} is given by
$$S_{0}:=C_{k} \|f \|_{L^1_{\g+3+2s}}^{\frac{1}{2}} \| \langle \cdot\rangle^{\frac{k+\g}{2}}\,\psi \|_{ \mathbb{H}^s } \left(\int_{\R^{3}} \bm{c}_\gamma [ \la \cdot \ra |f| ] \langle v\rangle^{k}\psi^2 \d v \right)^{\frac{1}{2}}\,,$$
and the integral is  estimated using again the $\e$-Poincar\'e inequality, Proposition \ref{prop:ePo}, to deduce that for all $\delta >0$,
\begin{multline*}
 \left(\int_{\R^{3}} \bm{c}_\gamma [ \la \cdot \ra |f| ] \langle v\rangle^{k}\psi^2 \d v \right)^{\frac{1}{2}}\\
\leq 
 \left(\delta\|\langle \cdot\rangle^{\frac{\g+k}{2}}\psi\|_{\mathbb{H}^{s}}^{2}+C_{0}\left( \|f\|_{L^{1}_{1}}  +\delta^{-\frac{\nu}{s-\nu}}\left\|\langle \cdot \rangle^{|\g|+1}f\right\|_{L^2}^{\frac{s}{s-\nu}}\right)\|\langle \cdot \rangle^{\frac{k}{2}}\psi\|_{L^2}^{2}\right)^{\frac{1}{2}}\\
\leq \frac{1}{2}\sqrt{\delta}\|\langle \cdot\rangle^{\frac{\g+k}{2}}\psi\|_{\mathbb{H}^{s}}+\frac{1}{2}\sqrt{C_{0}}\left( \|f\|_{L^{1}_{1}}  +\delta^{-\frac{\nu}{s-\nu}}\left\|\langle \cdot \rangle^{|\g|+1}f\right\|_{L^2}^{\frac{s}{s-\nu}}\right)^{\frac{1}{2}}\|\langle \cdot \rangle^{\frac{k}{2}}\psi\|_{L^2}.\end{multline*}
Thus,
\begin{multline*}
S_{0} \leq \frac{C_{k}}{2}\sqrt{\delta}\,\|f\|_{L^{1}_{\g+3+2s}}^{\frac{1}{2}}\|\langle \cdot\rangle^{\frac{\g+k}{2}}\psi\|_{\mathbb{H}^{s}}^{2}\\+\frac{1}{2}\sqrt{C_{0}}\|f\|_{L^{1}_{\g+3+2s}}^{\frac{1}{2}}\left( \|f\|_{L^{1}_{1}}  +\delta^{-\frac{\nu}{s-\nu}}\left\|\langle \cdot \rangle^{|\g|+1}f\right\|_{L^2}^{\frac{s}{s-\nu}}\right)^{\frac{1}{2}}\|\langle \cdot \rangle^{\frac{k}{2}}\psi\|_{L^2}\|\langle \cdot\rangle^{\frac{\g+k}{2}}\psi\|_{\mathbb{H}^{s}}.
\end{multline*}
Choosing $\delta=2C_{k}^{-1}\e^{2}$ and using Young's inequality, we deduce that
$$S_{0} \leq \e\|f\|_{L^{1}_{\g+3+2s}}^{\frac{1}{2}}\|\langle \cdot\rangle^{\frac{\g+k}{2}}\psi\|_{\mathbb{H}^{s}}^{2} + \tilde{C}_{\e}\|f\|_{L^{1}_{\g+3+2s}}^{\frac{1}{2}}\left( \|f\|_{L^{1}_{1}}  +\delta^{-\frac{\nu}{s-\nu}}\left\|  f\right\|_{L^2_{ 2(|\g|+1) }}^{\frac{s}{s-\nu}}\right) \| \psi\|_{L^2_k}^{ 2}.$$
Gathering the estimates for $S$ and $S_{0}$, we deduce that
$$|\mathcal{R}_{k}(f,\psi,\psi)| \leq \varepsilon\left(1+\|f\|_{L^{1}_{\g+3+2s}}^{\frac{1}{2}}\right)\left\|\langle \cdot\rangle^{\frac{\g}{2}}\psi\right\|_{\mathbb{H}^{s}_k }^{2}
+\bm{\Lambda}_{\e}[f]\,\|  \psi\|_{L^2_k}^{2}\,,$$
where
$$\bm{\Lambda}_{\e}[f]=\bm{C}_{\e} \left(1+\|f\|_{L^{1}_{\g+3+2s}}^{\frac{1}{2}}\right)
\left(\|f\|_{L^{1}_{4+|\g|}}  + \left\|\langle \cdot \rangle^{\frac{\g+k}{2}}f\right\|_{{\mathbb{H}}^{s}}^{2} + \|f\|_{L^{2}_{2 (4+|\g|) }}^{\frac{s}{s-\nu}}\right)\,,$$
for some positive $\bm{C}_{\e} >0$ depending on $k,s,\g$ and $\nu=-\frac{3+2\g}{4}$.

\medskip
We proceed in the same way to estimate $|\mathcal{R}_{k}(\psi,g,\psi)|$. In this case, the sum $S' = S_1' + S_2'$ in \eqref{eq:comm_general} is given by
$$S_1' = C_k \left( \int_{\R^{3}} \bm{c}_\gamma [ \la \cdot \ra^4 |g| ] \langle v\rangle^{k - \gamma }\psi^2 \d v \right)^{ \frac{1}{2} } \left( \int_{\R^{3}} \bm{c}_\gamma [ \la \cdot \ra^4 |g| ] \langle v\rangle^{k }\psi^2 \d v \right)^{ \frac{1}{2} } \, ,$$
$$S_2' = C_k \left( \int_{\R^{3}} \bm{c}_\gamma [ \la \cdot \ra^4 |\psi| ] \langle v\rangle^{k - \gamma }g^2 \d v \right)^{ \frac{1}{2} } \left( \int_{\R^{3}} \bm{c}_\gamma [ \la \cdot \ra^4 |\psi| ] \langle v\rangle^{k }\psi^2 \d v \right)^{ \frac{1}{2} } \, .$$
The two integrals composing $S_2'$ are estimated using \eqref{eq:conv_soft} with $\beta = -\gamma$ and $\beta = 0$ respectively, which yields
\begin{equation*}\begin{split}
	S'_2 &\le C_k \| \la \cdot \ra^{4 + 2 | \gamma | + \ell } \psi \|_{L^2} \left( \| \la \cdot \ra^{\frac{\g+k}{2}} g \|_{ \mathbb{H}^s }^2 + \| g \|^2_{L^2_k} \right)^{\frac{1}{2}} \left( \| \la \cdot \ra^{\frac{\g+k}{2}} \psi \|_{ \mathbb{H}^s  }^2 + \| \psi \|^2_{L^2_k} \right)^{ \frac{1}{2} } \\
	&\le  \e \| \la \cdot \ra^{\frac{\g+k}{2}} \psi \|^2_{\mathbb{H}^s  } + \e \| \psi \|_{L^2_k}^2 + \frac{C_k^2}{4 \e} \left( \| \la \cdot \ra^{\frac{\g+k}{2}} g \|^2_{ \mathbb{H}^s  } + \| g \|_{L^2_k}^2 \right) \| \psi \|^2_{L^2_k} \,.\end{split} 
\end{equation*}
Now, the first integral in $S_1'$ is estimated using \eqref{eq:conv_soft} with $\beta = -\gamma$ while the second is estimated thanks to Proposition \ref{prop:ePo}.   We deduce that
\begin{multline*}
S_1'  \le   C_k \| \la \cdot \ra^{4 + 2 | \gamma | + \ell } g \|_{L^2}^{ \frac{1}{2} } \left( \| \la \cdot \ra^{\frac{\g+k}{2}} \psi \|_{ \mathbb{H}^s }^2 + \| \psi \|^2_{L^2_k} \right)^{\frac{1}{2}} \\
\times \left(  \delta \left\|\langle \cdot \rangle^{\frac{\g+k}{2}}\psi\right\|_{{\mathbb{H}}^{s}}^{2}
+C_{0}\left( \| g \|_{L^{1}_{4}}  +\delta^{-\frac{\nu}{s-\nu}}\left\|\langle \cdot \rangle^{|\g|+4} g \right\|_{L^2}^{\frac{s}{s-\nu}}\right)\| \la \cdot \ra^{ \frac{\g+k}{2} } \psi\|_{L^2}^{2} \right)^{\frac{1}{2} }\,,\end{multline*}
which, using Young's inequality, yields for any $\e >0$,
\begin{multline*}
S_{1}'\le  \widetilde{C}_{k} \left( \frac{\delta}{\e} + \e \| \la \cdot \ra^{4 + 2 |\gamma| + \ell } g\|_{L^2} \right) \| \la \cdot \ra^{\frac{\g+k}{2}} \psi \|_{ \mathbb{H}^s  }^2  \\
+ \widetilde{C}_{k, \e}  \left( { \| \la \cdot \ra^{4 + 2 |\gamma| + \ell } g\|_{L^2} } + \| g \|_{L^{1}_{4}} +\delta^{-\frac{\nu}{s-\nu}}\left\|\langle \cdot \rangle^{|\g|+4} g \right\|_{L^2}^{\frac{s}{s-\nu}} \right) \| \psi \|_{L^2_k}^2 \,.\end{multline*}
Taking $\delta = \e^2$, we therefore have, for $k > 11 + 4s$, that
\begin{multline*}
S' \leq \e \, \widetilde{C}_k  \left( 1 + \| g \|_{L^2_k} \right) \left\|\langle \cdot \rangle^{\frac{\g+k}{2}}\psi\right\|_{{\mathbb{H}}^{s} }^{2} \\
+ C_{\e}\,\left( { \| g\|_{L^2_{k}}+ } \|g\|_{L^{1}_{4+|\g|}}  + \left\|\langle \cdot \rangle^{\frac{\g+k}{2}}g\right\|_{{\mathbb{H}}^{s}}^{2} + \|g\|_{L^{2}_{2 (4+|\g|) }}^{\frac{s}{s-\nu}}\right)\,\| \psi\|_{L^2_k}^{2}.\end{multline*}
The third term in \eqref{eq:comm_general} is given by
$$S_{0}':=C_{k} \| \psi \|_{L^1_{\g+3+2s}}^{\frac{1}{2}} \| \la \cdot \ra^{\frac{k}{2}}g \|_{ \mathbb{H}^s} \left(\int \bm{c}_\gamma [ \la \cdot \ra | \psi| ] \langle v\rangle^{k}\psi^2 \d v \right)^{\frac{1}{2}}\,.$$
We estimate this integral thanks to \eqref{eq:conv_soft} with $\beta = 0$ to deduce that 
$$S_{0}' \lesssim \|\psi \|_{L^1_{\g+3+2s}}^{\frac{1}{2}} \| \la \cdot \ra^{\frac{k}{2}} g \|_{ \mathbb{H}^s} \| \la \cdot \ra^{|\g|+1+\l} \psi\|_{L^{2}}^{\frac{1}{2}}\|\langle\cdot\rangle^{\frac{\g+k}{2}}\psi\|_{\mathbb{H}^{s} }
\lesssim \| \psi\|_{L^2_k}\|\la \cdot\ra^{\frac{k}{2} }g \|_{ \mathbb{H}^s} \|\langle\cdot\rangle^{\frac{\g+k}{2}}\psi\|_{\mathbb{H}^{s}}\,,$$
as soon as $k > 5 + 4s + \max\{ 0, 4 + 2 \gamma \}$.  Then, using Young's inequality yields
$$|\mathcal{R}_{k - \gamma}(\psi,g, \la \cdot \ra^{ \frac{\gamma}{2} } \psi)| \leq \varepsilon \left( 1 + \| g \|_{L^2_k} \right) \left\|\langle \cdot\rangle^{\frac{\g+k}{2}}\psi\right\|_{\mathbb{H}^{s} }^{2}
+\bm{\Lambda}_{\e}[g]\,\| \psi\|_{L^2_k}^{2} \, .$$
The result is proved since the worst constraint is $k > 11 + 4s$.
\end{proof}
\subsection{Stability of solutions -- Proof of Theorem \ref{theo:unique}}
We are now in position to prove the main stability result for solutions to \eqref{eq:Boltz} under the Prodi-Serrin condition as stated in the Introduction.
 \begin{proof}[Proof of Theorem \ref{theo:unique}]
 We consider two solutions $h,g$ with   initial data $h(0)=h_{\rm in}$ and $g(0)=g_{\rm in}$,
 and write $f=g-h$. {We notice first that, since $h_{\rm in}, g_{\rm in} \in L^{2}_{k+|\g|}(\R^{d}) \cap L^{1}_{\ell}(\R^{d})$, estimate \eqref{eq:propaLp} implies that
\begin{equation*} 
h,g \in L^{\infty}\left([0,T]\,;\;L^{2}_{k+2|\g|}(\R^{d})\right) \qquad \text{ and } \qquad \la \cdot \ra^{\frac{k-\g}{2}}h, \la \cdot \ra^{\frac{k-\g}{2}}g \in L^{2}\left([0,T]\,;\,\mathbb{H}^{s}\right).
\end{equation*}
Since we also assumed $k >11+4s > 8+|\g|$, one checks easily with the definition \eqref{eq:Lamb} that
\begin{equation}\label{eq:propaL2}\bm{\Lambda}_{\e}[h] \in L^{1}([0,T]) \quad \text{ and } \quad \bm{\Lambda}_{\e}[g] \in L^{1}([0,T]), \quad \e >0.\end{equation}
} We are interested in the evolution of 
 $$\mathscr{N}(t)=\|h(t)-g(t)\|_{L^{2}_k }^{2}=\int_{\R^{d}}\left|f(t,v)\right|^{2}\langle v\rangle^{k}\d v.$$
It follows that
$$\frac{1}{2}\dfrac{\d}{\d t}\mathscr{N}(t)=\mathscr{N}_{1}(t)+\mathscr{N}_{2}(t)\,,$$
where
$$\mathscr{N}_{1}(t)=\int_{\R^{d}}\Q\left(g(t),f(t)\right)f(t,v)\langle v\rangle^{k}\d v, \qquad \mathscr{N}_{2}(t)=\int_{\R^{d}}\Q\left(f(t),h(t)\right)(v)f(t,v)\langle v\rangle^{k}\d v.$$
Setting
$$F(t,v)= \la v \ra^{ \frac{k}{2} } f(t,v) \qquad\text{and}\qquad H(t,v)=\la v \ra^{ \frac{k}{2} } h(t,v)\,,$$
one has from \eqref{eq:commutator_implicit_def}
$$\mathscr{N}_{1}(t)=\langle \Q(g(t),f(t)),f(t)\rangle_{L^{2}_k } =\langle \Q(g(t),F(t))\,,F(t)\rangle_{L^{2}} +\mathcal{R}_{k}(g(t),f(t),f(t))\,,$$
and
$$\mathscr{N}_{2}(t)=\langle \Q(f(t), \la \cdot \ra^{ -\frac{\g}{2} }  H (t)), \la \cdot \ra^{ \frac{\g}{2} } F(t)\rangle_{L^{2}}+\mathcal{R}_{k-\gamma}(f(t),h(t), \la \cdot \ra^{ \frac{\gamma}{2} } f(t)) \, .$$
We estimate the first term in $\mathscr{N}_{2}(t)$ using \cite[Eq. (2.1)]{amu11} and \cite[Proposition 2.1]{amu11}, both with $m = \ell = 0$, so that 
\begin{multline*}
\langle \Q(f(t), \la \cdot \ra^{ -\frac{\g}{2} } H(t)), \la \cdot \ra^{ \frac{\g}{2} }  F (t)\rangle_{L^{2}} \\
\leq C_{\gamma, s,b} \left( \|f(t)\|_{ L^{1}} +\|f(t)\|_{L^{2}}\right) \|\langle \cdot\rangle^{\frac{\g}{2}}F(t)\|_{\mathbb{H}^{s}}\|\langle \cdot\rangle^{-\frac{\g}{2}}H(t)\|_{\mathbb{H}^{s}}.\end{multline*}
A suitable use of Young's inequality allows to deduce that there is a constant $C >0$ such that
$$\langle \Q(f(t), \la \cdot \ra^{-\frac{\g}{2}} H(t)), \la \cdot \ra^{\frac{\g}{2}}  F(t)\rangle_{L^{2}} \leq \e\|\langle \cdot\rangle^{\frac{\g}{2}}F(t)\|_{\mathbb{H}^{s}}^{2}+\frac{C}{\e}\|F(t)\|_{L^2}^{2}\|\langle \cdot\rangle^{- \frac{\g}{2}}H(t)\|_{\mathbb{H}^{s}}^{2}$$
for any $\e >0$ and $t \in [0,T]$.  Together with the control of $\mathcal{R}_{k-\g}(f(t),  h(t), \la \cdot \ra^{\frac{\g}{2}} f(t))$ and $\mathcal{R}_{k}(g(t),f(t),f(t))$ in Corollary \ref{cor:Comm} gives
\begin{multline*}
|\mathscr{N}_{2}(t)| + |\mathcal{R}_{k}(g(t),f(t),f(t))| \leq  \e\left(2+C_{1}\right)\left\|\langle \cdot\rangle^{\frac{\g}{2}}F(t)\right\|_{\mathbb{H}^{s}}^{2}\\
+\left(\bm{\Lambda}_{\e}[h(t)] + \bm{\Lambda}_{\e}[g(t)] + C\e^{-1}\|\langle \cdot\rangle^{-\frac{\g}{2}}H(t)\|_{\mathbb{H}^{s}}^{2} 	\right)\,\|F(t)\|_{L^2}^{2}\,,\end{multline*}
where $C_{1}=\max_{t \in [0,T]}\sqrt{\lm_{\frac{k}{2}+3+2s}(t)}.$  It remains only to estimate the term
 \begin{equation*}\begin{split}
\langle \Q(g(t),F(t)),F(t)\rangle_{L^{2}}&=\int_{\R^{2d}\times \S^{d-1}}B(u,\sigma)g(t,v_{*})F(t,v)\left[F(t,v')-F(t,v)\right]\d v\d v_{*}\d \sigma\\
&=\frac{1}{2}\int_{\R^{2d}\times\S^{d-1}}B(u,\sigma)g(t,\vet)\left[F(t,v')^{2}-F(t,v)^{2}\right]\d\sigma\d v\d v_{*}\\
&\phantom{+++} -\frac{1}{2}\int_{\R^{2d}\times\S^{d-1}}B(u,\sigma)g(t,\vet)\left[F(t,v')-F(t,v)\right]^{2}\d\sigma\d v\d v_{*}\,.\end{split}\end{equation*}
Using the Cancellation Lemma \ref{lem:canc} and the coercivity estimate Lemma \ref{lem:coerc}, we deduce that
{\begin{multline*}
\langle \Q(g(t),F(t)),F(t)\rangle_{L^{2}} \leq -\frac{c_{0}}{2} \left\|\langle \cdot \rangle^{\frac{\g}{2}}F(t)\right\|_{{\mathbb{H}}^{s}}^{2} + C_{0}\,\left\|\langle \cdot\rangle^{\frac{\g}{2}}F(t)\right\|_{L^2}^{2} \\
 + \frac{1}{2}\|\tilde{b}\|_{L^{1}(\S^{d-1})}\int_{\R^{d}}F^{2}(t,v)\bm{c}_{\g}[g(t)](v)\d v\,,\end{multline*}
where we also used that $$-\|\langle \cdot \rangle^{\frac{\g}{2}}F(t)\|_{\dot{\mathbb{H}}^{s}}^{2}=-\|\langle \cdot \rangle^{\frac{\g}{2}}F(t)\|_{\mathbb{H}^{s}}^{2}+\|\langle \cdot \rangle^{\frac{\g}{2}}F(t)\|_{L^{2}}^{2}$$ to obtain a bound involving the full $\mathbb{H}^{s}$ norm.}   Thanks to $\e$-Poincar\'e inequality \eqref{eq:estimatc}, we deduce for any $\e >0$ that there exists a constant $C_{\e} >0$ such that
\begin{multline*}
\int_{\R^{d}}F^{2}(t,v)\bm{c}_{\g}[g(t)](v)\d v \leq \e\,\left\|\langle \cdot \rangle^{\frac{\g}{2}}F(t)\right\|_{ {{\mathbb{H}}^{s}}}^{2}
\\
+C_{\e}\left( \|g(t)\|_{L^1}  + {\left\|\langle \cdot \rangle^{|\g|}g(t)\right\|_{L^q}^{\frac{s}{s-\nu}}}\right)\int_{\R^{d}}F^{2}(t,v)\langle v\rangle^{\g}\d v\,.\end{multline*}
Gathering all these estimates and choosing  {$\e=\e_{0} >0$ sufficiently small, we observe that
$$
\dfrac{1}{2}\dfrac{\d}{\d t}\mathscr{N}(t) + \frac{c_{0}}{4} \left\|\langle \cdot \rangle^{\frac{\g}{2}}F(t)\right\|_{ {{\mathbb{H}}^{s}}}^{2} \leq 
{\Theta}(t) \left\|\langle \cdot\rangle^{\frac{\g}{2}}F(t)\right\|_{L^2}^{2}\,,$$
where
$${\Theta}(t)=C\left(1+\bm{\Lambda}_{\e_{0}}[h(t)]+\bm{\Lambda}_{\e_{0}}[g(t)]+\|\langle \cdot \rangle^{\frac{k-\g}{2}}h(t)\|_{\mathbb{H}^{s}}^{2}+\left\|\langle \cdot \rangle^{|\g|}g(t)\right\|_{L^q}^{\frac{s}{s-\nu}}\right)\,,$$}
with $C >0$ depending explicitly on 
$$\sup_{t \in [0,T]}\left(\|h(t)\|_{L^{1}_{\g+3+2s}}+\|f(t)\|_{L^{1}_{\g+3+2s}}+\|g(t)\|_{L^{1}_{\g+3+2s}}\right).$$
 {By assumption \eqref{eq:ProSergf}, recall that $r=\frac{s}{s-\nu}$, and thanks to \eqref{eq:propaL2} we conclude that $\Theta \in L^{1}([0,T])$}.  Gronwall Lemma gives the conclusion. \end{proof} 

\appendix

\section{Technical results}\label{app:tech}
The following elementary inequality can be found in \cite{ricardo}.
\begin{lem} For any $X,Y \geq0$ and any $p >1$, one has
\begin{equation}\label{eq:XY}
Y\left[X^{p-1}-Y^{p-1}\right] \leq \frac{1}{p'}\left[X^{p}-Y^{p}\right] - \frac{1}{\max(p,p')}\left[X^{\frac{p}{2}}-Y^{\frac{p}{2}}\right]^{2}.\end{equation}
Moreover, 
\begin{equation}\label{eq:X2}
Y\left|X^{p-1}-Y^{p-1}\right] \leq |X^{\frac{p}{2}}-Y^{\frac{p}{2}}|\,\left[X^{\frac{p}{2}}+Y^{\frac{p}{2}}\right].\end{equation}
 \end{lem}
\begin{proof} Observe that, see for example \cite[Lemma 1]{ricardo},
$$x^{\frac{2}{p'}}-1 \leq \frac{1}{p'}(x^{2}-1)-\frac{1}{\max(p,p')}(x-1)^{2}, \qquad \frac{1}{p}+\frac{1}{p'}=1\,,$$
holds for any $x \geq0$.  Observing then $Y\left[X^{p-1}-Y^{p-1}\right]=Y^{p}\left[x^{\frac{2}{p'}}-1\right]$, with $x=\left(\frac{X}{Y}\right)^{\frac{p}{2}}$ one deduces the result. Moreover, since $|x^{\frac{2}{p'}}-1|\leq |x^{2}-1|$ for any $x\geq0$, it follows that
\begin{equation*} 
Y\left|X^{p-1}-Y^{p-1}\right|=Y^{p}\left|\left(\frac{X^{\frac{p}{2}}}{Y^{\frac{p}{2}}}\right)^{\frac{2}{p'}}-1\right|
\leq Y^{p}\left|\left(\frac{X^{\frac{p}{2}}}{Y^{\frac{p}{2}}}\right)^{2}-1\right|=|X^{p}-Y^{p}|\,, \end{equation*}
which gives \eqref{eq:X2}.
\end{proof}
We end this section with an estimate for difference of $\varphi(v')-\varphi(v)$ for polynomially growing mappings $\varphi$. 
\begin{lem}\label{lem:cancel} Assume that $d\in\N$, $d\ge 2$.  For any $\ell \geq0$ and $\alpha \in (0,1]$, there exists $C_{\ell,\alpha} > 0$ depending only on $\ell,d,\alpha$ such that, for any $v,\vet \in \R^{d},$ $\sigma \in \S^{d-1}$, 
\begin{equation}
\label{eq:weight_cancellation}
\left| \la v \ra^\ell - \la v' \ra^\ell \right| \leq C_{\ell,\alpha} \sin\left( \tfrac{\theta}{2} \right)|v-\vet|^{\alpha}\left(\la v\ra^{\ell-\alpha}+\la \vet\ra^{\ell-\alpha}\right)\,,
\end{equation}
where we recall that $\cos \theta=\frac{v-\vet}{|v-\vet|}\cdot \sigma$.  Moreover, for $\ell \geq 2$, 
\begin{multline}
\label{eq:cancel-int}
\left|\int_{\S^{d-1}}\left( \la v \ra^\ell - \la v' \ra^\ell \right)\,b(\cos\theta)\d\sigma\right| \\
\leq C_{\ell}|v-\vet|^{2\alpha}\left(\langle v\rangle^{\ell-2\alpha}+\langle \vet\rangle^{\ell-2\alpha}\right)\int_{0}^{\pi}b(\cos\theta)\sin^{d}\theta\d \sigma
\end{multline}
holds for any nonnegative and measurable $b\::\:[-1,1] \to \R^{+}$.
\end{lem}
\begin{proof} The first estimate is essentially established in \cite[Lemma 2.3]{amu10}. Indeed, as shown therein, 
$$|\la v\ra^{\ell}-\la v'\ra^{\ell}| \leq C \sin\left( \tfrac{\theta}{2} \right)|v-\vet|\left(\la v\ra^{\ell-1}+\la v'\ra^{\ell-1}\right)$$
and, since $\la v'\ra \leq \sqrt{\la v\ra^{2}+\la \vet\ra^{2}} \leq \frac{1}{2}\la v\ra+\frac{1}{2}\la \vet\ra$, this gives  \eqref{eq:weight_cancellation} for $\alpha=1$. For $\alpha < 1$, we easily deduce the result from 
$$|v-\vet| \leq |v-\vet|^{\alpha}\max\big(\la v\ra^{1-\alpha},\la \vet\ra^{1-\alpha}\big).$$
For the second estimate, we use a suitable parametrisation of $\S^{d-1}$ given in \cite{carlen} where, given $\bm{u}=\frac{v-\vet}{|v-\vet|}$ (for $v\neq \vet$) one writes $\sigma \in \S^{d-1}$ as
$$\sigma=\cos\theta\,\bm{u}+\sin \theta\,\omega, \qquad \omega \in \S^{d-2}(\bm{u}), \qquad \theta \in [0,\pi]\,,$$
where $\S^{d-2}(\bm{u})=\left\{\omega \in \S^{d-1}\,,\,\omega \cdot \bm{u}=0\right\}$.  Then, using \cite[Lemma 2.1]{carlen}, see also \cite[Lemma 4]{ricardo}, for any smooth $\varphi$
$$\left|\int_{\S^{d-2}(\bm{u})}\left(\varphi(v)-\varphi(v')\right)\d \omega \right|\leq |\S^{d-2}(\bm{u})|\sup_{|x| \leq\sqrt{|v|^{2}+|\vet|^{2}}}|\partial^{2}\varphi(x)||v-\vet|^{2}\sin^{2}\theta \qquad \forall \theta \in [0,\pi].$$
With $\varphi(x)=\langle x\rangle^{\ell},$ $\ell \geq 2$, since $|\partial^{2}\varphi(x)| \leq \ell^{2}\langle x\rangle^{\ell-2}$, one deduces that there is a constant $C_{\ell} >0$ depending on $\ell$ and $d$ only such that
$$\left|\int_{\S^{d-2}(\bm{u})}\left(\la v \ra^\ell - \la v' \ra^\ell\right)\d \omega \right|\leq C_{\ell}\left(\langle v\rangle^{\ell-2}+\langle \vet\rangle^{\ell-2}\right)|v-\vet|^{2}\sin^{2}\theta \qquad \forall \theta \in [0,\pi].$$
This gives \eqref{eq:cancel-int} for $\alpha=2$ after additional integration over $\theta$ since $\d\sigma=\sin^{d-2}\theta\d\theta\d\omega$. One goes then from $\alpha=2$ to any $\alpha \in (0,2]$ as in the proof of \eqref{eq:weight_cancellation}.
\end{proof}
We give here the a proof of Prop. \ref{prop:AMSY} adapted from \cite[Prop. 3.1]{AMSY}. Recall the following lemma established in \cite[Lemma 2.5]{AMSY}. 
\begin{lem}
Let $(v,\vet,\sigma) \in \R^{3}\times \R^{3} \times \S^{2}$ be fixed and let $(v',\vet')$ be the associated post-collisional velocities. Given $\l >6$,  one has that
\begin{equation}\label{eq:Lem25}\begin{split}
\langle v'\rangle^{\l}-\langle v\rangle^{\l}&=\l\langle v\rangle^{\l-2}\,|v-\vet|\left(v\cdot \omega\right)\cos^{\l-1}\left(\tfrac{\theta}{2}\right)\sin\tfrac{\theta}{2}+\langle \vet \rangle^{\l}\sin^{\l}\left(\tfrac{\theta}{2}\right)\\
&+\bm{r}_{1}+\bm{r}_{2}+\bm{r}_{3}+\langle v\rangle^{\l}\left(\cos^{\l}\left(\tfrac{\theta}{2}\right)-1\right),\end{split}
\end{equation}
where 
$$\omega=\frac{\sigma-\left(\sigma\cdot \widehat{u}\right)u}{\left|\sigma-\left(\sigma\cdot \widehat{u}\right)\widehat{u}\right|}, \qquad \widehat{u}=\frac{v-\vet}{|v-\vet|}\,,$$ and there exists $C_{\l} >0$ such that
\begin{equation}\label{eq:r1}
\begin{cases}
\left|\bm{r}_{1}\right| &\leq C_{\l}\langle v\rangle\,\langle \vet\rangle^{\l-1}\sin^{\l-3}\left(\tfrac{\theta}{2}\right)\,,\qquad 
\left|\bm{r}_{2}\right| \leq C_{\l}\langle v\rangle^{\l-2}\langle \vet\rangle^{2}\sin^{2}\left(\tfrac{\theta}{2}\right),\\
\left|\bm{r}_{3}\right| &\leq C_{\l}\langle v\rangle^{\l-4}\langle \vet\rangle^{4}\sin^{2}\left(\tfrac{\theta}{2}\right)\,.
\end{cases}\end{equation}
\end{lem}
\subsection{Proof of Proposition \ref{prop:AMSY}}\label{app:AMSY}
With the notations introduced earlier, we can prove Prop. \ref{prop:AMSY}.
\begin{proof}[Proof of Proposition \ref{prop:AMSY}]With the decomposition of $\langle v'\rangle^{\l}-\langle v\rangle^{\l}$ in \eqref{eq:Lem25}, and with $\d\mu=\d v\d\vet\d\sigma$, one has that
$$
\int_{\R^{6}\times \S^{2}}b(\cos\theta)|v-\vet|^{\g}\left(\langle v'\rangle^{\l}-\langle v\rangle^{\l}\right)f_{\ast} g\,h'\d \mu= \sum_{j=1}^{6}\Gamma_{j}\,,
$$
with
$$\Gamma_{1}:=\l\int_{\R^{6}\times \S^{2}}b(\cos\theta)|v-\vet|^{\g+1}\langle v\rangle^{\l-2}\left(v\cdot \omega\right)\cos^{\l-1}\left(\tfrac{\theta}{2}\right)\sin\tfrac{\theta}{2}f_{\ast}g\,h'\d\mu,$$
$$\Gamma_{2}=\int_{\R^{6}\times \S^{2}}b(\cos\theta)|v-\vet|^{\g}\langle \vet\rangle^{\l}\sin^{\l}\left(\tfrac{\theta}{2}\right)f_{\ast}g\,h'\d\mu\,,$$
$$\Gamma_{2+k}=\int_{\R^{6}\times \S^{2}}b(\cos\theta)|v-\vet|^{\g} \bm{r}_{k}f_{\ast}gh'\d\mu, \quad k=1,2,3,$$
and
$$\Gamma_{6}=\int_{\R^{6}\times \S^{2}}b(\cos\theta)|v-\vet|^{\g}\left(\cos^{\l}\left(\tfrac{\theta}{2}\right)-1\right)\langle v\rangle^{\l}f_{\ast}g\,h'\d\mu.$$ 
Using Cauchy-Schwarz inequality, one has that, for $\alpha+\beta=2\l$,
\begin{multline*}
\left|\Gamma_{2}\right| \leq 
\left(\int_{\R^{6}\times \S^{2}}b(\cos\theta)|v-\vet|^{\g}\langle \vet\rangle^{2\l}f_{\ast}^{2}\sin^{\alpha}\left(\tfrac{\theta}{2}\right) \,|g|\d\mu\right)^{\frac{1}{2}}\\
\left(\int_{\R^{6}\times \S^{2}}b(\cos\theta)|v-\vet|^{\g} \sin^{\beta}\left(\tfrac{\theta}{2}\right)|g|\,|h'|^{2}\d\mu\right)^{\frac{1}{2}}.
\end{multline*}
Using now  the singular change of variable (see \cite{ADVW})
\begin{equation}\label{eq:sing}
\int_{\R^{3}\times \S^{2}}b(\cos\theta)|v-\vet|^{\g}\varphi(v')\d \sigma\d\vet=\int_{\R^{3}\times\S^{2}}|v-\vet|^{\g}b(\cos\theta)\sin^{-3-\g}\left(\tfrac{\theta}{2}\right)\varphi(\vet)\d\sigma\d\vet,
\end{equation} 
and choosing $\alpha=\beta-3-\g=\l-\frac{3+\g}{2}$, we deduce that 
\begin{multline*}
\left|\Gamma_{2}\right| \leq 
\left(\int_{\R^{6}\times \S^{2}}b(\cos\theta)|v-\vet|^{\g}\langle \vet\rangle^{2\l}f_{\ast}^{2}\sin^{\l-\frac{3+\g}{2}}\left(\tfrac{\theta}{2}\right) \,|g|\d\mu\right)^{\frac{1}{2}}\\
\left(\int_{\R^{6}\times \S^{2}}b(\cos\theta)|v-\vet|^{\g}\,\sin^{\l-\frac{3+\g}{2}}\left(\tfrac{\theta}{2}\right)|g|\,|h_{\ast}|^{2}\d\mu\right)^{\frac{1}{2}}.
\end{multline*}
Using Cauchy-Schwarz inequality yields
\begin{multline*}
\left|\Gamma_{2}\right| \leq  
\left(\int_{\S^{2}}b(\cos\theta)  \sin^{\l-\frac{3+\g}{2}}\left(\tfrac{\theta}{2}\right)\d\sigma\right)
\left(\int_{\R^{6}} |v-\vet|^{\g}\langle \vet\rangle^{2\l}f_{\ast}^{2}\,|g|\d\vet\d v\right)^{\frac{1}{2}}\\
\left(\int_{\R^{6}} |v-\vet|^{\g}  |g|\,|h_{\ast}|^{2}\d\vet\d v\right)^{\frac{1}{2}}.
\end{multline*}
Assuming $\l > \frac{3+\g}{2}+2s$, one has that
$$\int_{\S^{2}}b(\cos\theta)  \sin^{\l-\frac{3+\g}{2}}\left(\tfrac{\theta}{2}\right)\d\sigma=\bm{C}_{2,\l}(b) < \infty\,,$$
and
$$|\Gamma_{2}| \leq \bm{C}_{2,\l}(b)\left(\int_{\R^{3}}\bm{c}_{\g}[|g|](v)\left(\langle v\rangle^{\l}f(v)\right)^{2}\d v\right)^{\frac{1}{2}}\left(\int_{\R^{3}}\bm{c}_{\g}[|g|](v)\,h^{2}(v)\d v\right)^{\frac{1}{2}}.$$
In the same way, using \eqref{eq:r1},
\begin{multline*}
\left|\Gamma_{3}\right| \leq C_{\l}\int_{\R^{6}\times \S^{2}}b(\cos\theta)|v-\vet|^{\g} \langle v\rangle\,\langle \vet\rangle^{\l-1}\sin^{\l-3}\left(\tfrac{\theta}{2}\right) |f_{\ast}gh'|\d\mu\\
\leq C_{\l}\left(\int_{\R^{6}\times\S^{2}}b(\cos\theta) {|v-\vet|^{\g}}\sin^{\l-\frac{9+\g}{2}}\left(\tfrac{\theta}{2}\right)\langle v\rangle\,|g|\left(\langle \vet\rangle^{\l-1}f_{\ast}\right)^{2}\d\mu\right)^{\frac{1}{2}}\\
\left(\int_{\R^{6}\times\S^{2}}{|v-\vet|^{\g}}b(\cos\theta)\sin^{\l-\frac{9+\g}{2}}\left(\tfrac{\theta}{2}\right)\langle v\rangle\,|g|\,|h_{\ast}|^{2}\d\mu\right)^{\frac{1}{2}}\,.\end{multline*}
With $\l > \frac{9+\g}{2}+2s$, one has that
$$\int_{\S^{2}}b(\cos\theta)\sin^{\l-\frac{9+\g}{2}}\left(\tfrac{\theta}{2}\right)\d\sigma=\bm{C}_{3,\l}(b) <\infty\,,$$
and 
$$\left|\Gamma_{3}\right| \leq \bm{C}_{3,\l}(b)\left(\int_{\R^{3}}\bm{c}_{\g}\left[\langle \cdot\rangle\,|g|\right]\left(\langle \cdot\rangle^{\l-1}f\right)^{2}\d v\right)^{\frac{1}{2}}
\left(\int_{\R^{3}}\bm{c}_{\g}\left[\langle \cdot\rangle\,|g|\right]\,h^{2}\d v\right)^{\frac{1}{2}}.$$
For $\Gamma_{4}$, using the bound for $|\bm{r}_{2}|$ in \eqref{eq:r1}, it holds that
$$|\Gamma_{4}| \leq C_{\l}\int_{\R^{6}\times\S^{2}}b(\cos\theta)\sin^{2}\left(\tfrac{\theta}{2}\right)|v-\vet|^{\g}\langle\vet\rangle^{2}|f_{\ast}|\langle v\rangle^{\l-2}|g|\,|h'|\d\mu\,,$$
so that
\begin{multline*}
|\Gamma_{4}|\leq C_{\l}\left(\int_{\R^{6}\times\S^{2}}\langle \vet\rangle^{2}|f_{\ast}|b(\cos\theta)\sin^{2}\left(\tfrac{\theta}{2}\right)|v-\vet|^{\g}\left(\langle v\rangle^{\l-2}g(v)\right)^{2}\d\mu\right)^{\frac{1}{2}}\\
\left(\int_{\R^{6}\times\S^{2}}\langle \vet\rangle^{2}|f_{\ast}|b(\cos\theta)\sin^{2}\left(\tfrac{\theta}{2}\right)|v-\vet|^{\g}\left(h(v')\right)^{2}\d\mu\right)^{\frac{1}{2}}\,.
\end{multline*}
Notice that
$$\int_{\S^{2}}b(\cos\theta)\sin^{2}\left(\tfrac{\theta}{2}\right)\d\sigma=\bm{C}(b) < \infty\,,$$
from which we conclude that
\begin{multline*}
\int_{\R^{6}\times\S^{2}}\langle \vet\rangle^{2}|f_{\ast}|b(\cos\theta)\sin^{2}\left(\tfrac{\theta}{2}\right)|v-\vet|^{\g}\left(\langle v\rangle^{\l-2}g(v)\right)^{2}\d\mu \\
\leq \bm{C}(b)\int_{\R^{3}}\bm{c}_{\g}\left[\langle \cdot\rangle^{2}|f|\right]\left(\langle\cdot\rangle^{\l-2}g\right)^{2}\d v.\end{multline*}
In the same way, using the regular change of variable (see \cite{ADVW})
\begin{equation}\label{eq:reg}
\int_{\R^{3}\times \S^{2}}b(\cos\theta)|v-\vet|^{\g}\varphi(v')\d \sigma\d v=\int_{\R^{3}\times\S^{2}}{|v-\vet|^{\g}}b(\cos\theta)\cos^{-3-\g}\left(\tfrac{\theta}{2}\right)\varphi(v)\d\sigma\d v\,,
\end{equation} 
and since
$$\int_{\S^{2}}b(\cos \theta)\sin^{2}\left(\tfrac{\theta}{2}\right)\cos^{-3-\g}\left(\tfrac{\theta}{2}\right)\d\sigma \leq {\bm{C}}(b) < \infty\,,$$
we deduce that
$$\int_{\R^{6}\times\S^{2}}\langle \vet\rangle^{2}|f_{\ast}|b(\cos\theta)\sin^{2}\left(\tfrac{\theta}{2}\right)|v-\vet|^{\g}\left(h(v')\right)^{2}\d\mu 
\leq \tilde{\bm{C}}(b)\int_{\R^{3}}\bm{c}_{\g}\left[\langle \cdot\rangle^{2}|f|\right]h^{2}\d v\,,$$
where, with $\bm{C}_{4,\l}(b)=C_{\l}\bm{C}(b)$, it holds
$$|\Gamma_{4}| \leq \bm{C}_{4,\l}(b) \left(\int_{\R^{3}}\bm{c}_{\g}\left[\langle \cdot\rangle^{2}|f|\right]\left(\langle\cdot\rangle^{\l-2}g\right)^{2}\d v\right)^{\frac{1}{2}}\,\left(\int_{\R^{3}}\bm{c}_{\g}\left[\langle \cdot\rangle^{2}|f|\right]h^{2}\d v\right)^{\frac{1}{2}}.$$
The exact same computations show that 
$$|\Gamma_{5}| \leq \bm{C}_{5,\l}(b)\left(\int_{\R^{3}}\bm{c}_{\g}\left[\langle \cdot\rangle^{4}|f|\right]\left(\langle\cdot\rangle^{\l-4}g\right)^{2}\d v\right)^{\frac{1}{2}}\,\left(\int_{\R^{3}}\bm{c}_{\g}\left[\langle \cdot\rangle^{4}|f|\right]h^{2}\d v\right)^{\frac{1}{2}}.$$
The term $\Gamma_{6}$ is estimated similarily since
$$\int_{\S^{2}}b(\cos\theta)\left|1-\cos^{\l}\left(\tfrac{\theta}{2}\right)\right| \d\sigma \leq \bm{C}_{6,\l}(b),$$
consequently,
$$|\Gamma_{6}| \leq \bm{C}_{6,\l}(b)\left(\int_{\R^{3}}\bm{c}_{\g}[|f|]\left(\langle \cdot\rangle^{\l}g\right)^{2}\d v\right)^{\frac{1}{2}}\left(\int_{\R^{3}}\bm{c}_{\g}[|f|]h^{2}\d v\right)^{\frac{1}{2}}.$$
The estimate for $\Gamma_{1}$ is more delicate and one splits, as in \cite[Proposition 3.1]{AMSY},
$$\Gamma_{1}=\Gamma_{1,1}+\Gamma_{1,2}$$
where
$$\Gamma_{1,1}=\l \int_{\R^{6} \times \S^{2}}b(\cos\theta)|v-\vet|^{\g+1} \langle v\rangle^{\l-2} \left(\vet\cdot \tilde{\omega}\right)\cos^{\l}\left(\tfrac{\theta}{2}\right)\sin\tfrac{\theta}{2}\,f_{\ast}g h'\d\mu,$$
$$\Gamma_{1,2}=\l \int_{\R^{6} \times \S^{2}}b(\cos\theta)|v-\vet|^{\g+1}\langle v\rangle^{\l-2}\left(\vet\cdot \frac{v'-\vet}{|v'-\vet|}\right)\cos^{\l-1}\left(\tfrac{\theta}{2}\right)\sin^{2}\left(\tfrac{\theta}{2}\right)f_{\ast} g h' \d\mu,$$
with $\tilde{\omega}=\frac{v'-v}{|v'-v|}$. To estimate $\Gamma_{1,2}$, one first notices that
$$|\Gamma_{1,2}| \leq \l  \int_{\R^{6} \times \S^{2}}b(\cos\theta)\cos^{\l-1}\left(\tfrac{\theta}{2}\right)\sin^{2}\left(\tfrac{\theta}{2}\right)|v-\vet|^{\g}\langle \vet\rangle^{2} |f_{\ast}| \langle v\rangle^{\l-1}|g|\,|h'|\d\mu.$$
Then, observing that
$$\l\int_{\S^{2}}b(\cos \theta)\cos^{\l-1}\left(\tfrac{\theta}{2}\right)\sin^{2}\left(\tfrac{\theta}{2}\right)\d\sigma =\bm{C}_{1,\l}(b) < \infty,$$
a simple use of Cauchy-Schwarz inequality and the regular change of variable \eqref{eq:reg} give
$$|\Gamma_{1,2}| \leq \bm{C}_{1,\l}(b) \left(\int_{\R^{3}}\bm{c}_{\g}[\langle \cdot \rangle^{2} |f|]\left(\langle v\rangle^{\l-1} g\right)^{2}\d v\right)^{\frac{1}{2}} \left(\int_{\R^{3}}\bm{c}_{\g}[\langle \cdot \rangle^{2} |f|]h^{2}\d v\right)^{\frac{1}{2}}.$$
Now,  as in \cite[Proposition 3.1]{AMSY}, one has that
$$\Gamma_{1,1}=\l \int_{\R^{6} \times \S^{2}}b(\cos\theta)|v-\vet|^{\g+1} \left(\vet\cdot \tilde{\omega}\right)\cos^{\l}\left(\tfrac{\theta}{2}\right)\sin\tfrac{\theta}{2}\,f_{\ast}\left[G-G'\right] h'\d\mu,$$
where $G=\langle \cdot \rangle^{\l-2}g$ so that
$$|\Gamma_{1,1}|\leq \l \int_{\R^{6} \times \S^{2}}b(\cos\theta)|v-\vet|^{\g+1} \sin\tfrac{\theta}{2}\,\left( \langle \vet\rangle |f_{\ast}|\right)\left|G-G'\right|\, |h'|\d\mu.$$
Using again Cauchy-Schwarz inequality
\begin{multline*}
|\Gamma_{1,1}| \leq \l \left(\int_{\R^{6} \times \S^{2}}b(\cos\theta)|v-\vet|^{\g+2}\left( \langle \vet\rangle |f_{\ast}|\right)\left|G-G'\right|^{2}\d\mu\right)^{\frac{1}{2}}\\
\left(\int_{\R^{6} \times \S^{2}}b(\cos\theta)|v-\vet|^{\g}\sin^{2}\left(\tfrac{\theta}{2}\right)\left( \langle \vet\rangle |f_{\ast}|\right)|h'|^{2}\d\mu\right)^{\frac{1}{2}}.
\end{multline*}
Using the regular change of variables \eqref{eq:reg} in the last integral and the fact that
$$\int_{\S^{2}}b(\cos\theta)\sin^{2}\left(\tfrac{\theta}{2}\right)\d\sigma \leq \bm{C}(b),$$
we deduce, with the notations introduced in Remark \ref{rmq:Dgamma},  that
$$|\Gamma_{1,1}| \leq \l \sqrt{\bm{C}(b)}\,\sqrt{\mathscr{D}_{\g+2}\left[\langle \cdot \rangle |f|\,,\,G\right]}\left(\int_{\R^{3}}\bm{c}_{\g}[\langle \cdot \rangle\,|f|]h^{2}\d v\right)^{\frac{1}{2}}.$$
Combining these estimates, identifying the largest terms in the various estimates, yield the result.
\end{proof}

\end{document}